\documentclass[11pt]{article}
\usepackage{amscd, amsmath, amssymb, amsthm}
\usepackage[all,cmtip]{xy}
\usepackage[pagebackref]{hyperref}

\title{The motive of a classifying space}
\author{Burt Totaro}
\date{  }

\def\Z{\text{\bf Z}}
\def\Q{\text{\bf Q}}

\def\C{\text{\bf C}}
\def\P{\text{\bf P}}
\def\F{\text{\bf F}}

\def\arrow{\rightarrow}
\def\inj{\hookrightarrow}
\def\imp{\Rightarrow}
\def\surj{\twoheadrightarrow}

\def\Tor{\text{Tor}}
\def\Hom{\text{Hom}}
\def\HHom{\underline{\text{Hom}}}
\newcommand{\dlim}{\mathop{\varinjlim}\limits}
\def\gm{\text{gm}}
\def\hocolim{\text{hocolim}}
\def\holim{\text{holim}}
\def\Spec{\text{Spec}}
\def\tr{\text{tr}}

\def\nr{\text{nr}}
\def\Gal{\text{Gal}}
\def\bir{\text{bir}}
\def\perf{\text{perf}}
\def\dim{\text{dim}}
\def\codim{\text{codim}}
\def\Corr{\text{Corr}}
\def\Var{\text{Var}}
\def\cone{\text{cone}}
\def\eff{\text{eff}}

\begin{document}
\maketitle
\newtheorem{theorem}{Theorem}[section]
\newtheorem{corollary}[theorem]{Corollary}
\newtheorem{lemma}[theorem]{Lemma}

\theoremstyle{definition}
\newtheorem{definition}[theorem]{Definition}
\newtheorem{example}[theorem]{Example}

\theoremstyle{remark}
\newtheorem{remark}[theorem]{Remark}

The Chow group of algebraic cycles generally does not
satisfy the K\"unneth formula. Nonetheless, there are some
schemes $X$ over a field $k$ which
satisfy the {\it Chow K\"unneth property }that the product
$CH_*X\otimes_{\Z}CH_*Y \arrow CH_*(X\times_k Y)$ is an isomorphism
for all separated schemes $Y$ of finite type over $k$.
The Chow K\"unneth property implies
the {\it weak Chow K\"unneth property }that $CH_*X\arrow CH_*(X_F)$
is surjective for every finitely generated field $F$ over $k$
(or, equivalently, for every extension field $F$ of $k$).
We characterize several properties
of this type. (We also prove versions of all our results with
coefficients in a given commutative ring.)

Our characterizations of K\"unneth properties
are: first, a smooth proper scheme $X$ over $k$ satisfies
the weak Chow K\"unneth property
if and only if the Chow motive
of $X$ is a summand of a direct sum of {\it Tate motives}
(Theorem \ref{smoothproper}).
(This is related to known
results by Bloch, Srinivas, Jannsen, Kimura, and others.)
A more novel result is about an arbitrary separated scheme $X$
of finite type over $k$. We say that $X$
satisfies the {\it motivic K\"unneth property}
if the K\"unneth spectral sequence converges
to the motivic homology groups of $X\times_k Y$ for all $Y$. (Motivic
homology groups are also called {\it higher Chow groups}; they include
the usual Chow groups as a special case.)
We show that a $k$-scheme $X$ satisfies the motivic K\"unneth property
if and only if the motive of $X$ in Voevodsky's derived category
of motives is a {\it mixed Tate motive} (Theorem \ref{main}).
(An example of a scheme with these properties is any {\it linear scheme},
as discussed in section \ref{motives}.)
Finally, if a smooth but
not necessarily proper $k$-variety $X$
satisfies the weak Chow K\"unneth property,
then the {\it birational motive }of $X$
in the sense of Rost and Kahn-Sujatha is isomorphic to the birational
motive of a point (Corollary \ref{biratcor}).

The last result cannot be 
strengthened to say that the motive of $X$ is mixed Tate; one has
to consider motivic homology groups to get that conclusion. For example,
the complement $X$
of a curve of genus 1 in the affine plane has the Chow K\"unneth
property, since $CH_{i+2}(X\times_k Y)\cong CH_iY$
for all separated $k$-schemes $Y$ of finite type and all $i$;
but the motive of $X$ is not mixed Tate.

As an application of these general results, we disprove the
weak Chow K\"unneth property for some classifying spaces $BG$.
For an affine group scheme $G$ of finite type over a field $k$,
Morel-Voevodsky and I constructed
$BG$ as a direct limit of smooth $k$-varieties,
quotients by $G$ of open subsets of representations of $G$ over $k$
\cite[section 4.2]{MV},
\cite{Totarochow, Totarobook}. As a result,
the Chow ring of $BG$ makes sense. The Chow ring of $BG$ tensored
with the rationals is easy to compute; for example, if $G$ is finite,
then $CH^i(BG)\otimes\Q=0$ for $i>0$.
The challenge
is to understand the integral or mod $p$ Chow ring of $BG$.

For many finite groups $G$ and fields
$k$, the classifying
space $BG$ over $k$ satisfies the Chow K\"unneth property that
$CH^*BG\otimes_{\Z}CH_*Y\cong CH_*(BG\times_k Y)$
for all separated $k$-schemes
$Y$ of finite type. For example, 
an abelian $p$-group $G$ of exponent $e$ has the Chow K\"unneth property
when $k$ is a field of characteristic not
$p$ that contains the $e$th roots of unity. The Chow K\"unneth
property also holds for many other groups,
such as wreath products of abelian groups \cite[Lemma 2.12]{Totarobook}.
As a result,
\cite[Chapter 17]{Totarobook} asked whether every finite group $G$
has the Chow K\"unneth property over a field $k$ which contains
enough roots of unity. This would imply the weak Chow K\"unneth
property that $CH^*BG_k\arrow CH^*BG_F$ is surjective for
every extension field $F$ of $k$.

In this paper, we give the first examples of finite groups
for which the Chow K\"unneth property fails. For any finite group $G$ 
such that
$BG$ has nontrivial unramified cohomology, there is a finitely generated
field $F$ over $\overline{\Q}$ such that $CH^*BG_{\overline{\Q}}
\arrow CH^*BG_F$ is not surjective (Corollary \ref{BG}).
We also find a field $E$ containing $\overline{\Q}$ such that
$CH^i(BG_E)/p$ is infinite for some $i$ and some prime number $p$
(Corollary \ref{infinite}); this answers
another question in \cite[Chapter 18]{Totarobook}. In particular,
the ring $CH^*(BG_E)/p$ is not noetherian in such an example.

As recalled in section \ref{birat}, there are groups
of order $p^5$ for any
odd prime $p$, and groups of order $2^6$, which have nontrivial
unramified cohomology. This is surprisingly sharp.
In fact, the Chow ring
$CH^*BG_k$ of a $p$-group $G$ is independent of the field $k$ containing
$\overline{\Q}$, and consists
of transferred Euler classes of representations, when $G$ is
a $p$-group of order at most $p^4$
\cite[Theorem 11.1, Theorem 17.4]{Totarobook}. Moreover,
the weak Chow K\"unneth property holds
for all groups of order $2^5$ (Theorem \ref{order32}).

Finally, section \ref{quotientsect}
defines the {\it compactly supported motive},
in Voevodsky's derived category of motives, for a
quotient stack over a field. In particular, we get a notion
of the compactly supported motive $M^c(BG)$ for an affine group scheme $G$.
Once we have this definition, we can ask when $M^c(BG)$ is a mixed
Tate motive. This property is equivalent to the motivic K\"unneth
property for $BG$, and so it implies the Chow K\"unneth property
for $BG$. In particular, $BG$ is not mixed Tate for the groups
of order $p^5$ discussed above. On the other hand, we show that
many familiar finite groups, such as
the finite general linear groups in cross-characteristic
and the symmetric groups,
are mixed Tate (Theorems \ref{symmetric} and \ref{gl}).

The introduction to section \ref{TateBGsect}
discusses six properties of finite groups. It would be interesting
to find out whether all six properties are equivalent,
as the known examples suggest. The properties are:
stable rationality of $BG$ (say, over the complex numbers), meaning
stable rationality of quotient varieties $V/G$; triviality
for the birational motive of $BG$ (or equivalently, of quotient
varieties $V/G$);
Ekedahl's class of $BG$
in the Grothendieck ring of varieties being equal to 1
\cite{Ekedahl}; the weak
Chow K\"unneth property of $BG$; the Chow K\"unneth property of $BG$;
and the mixed Tate property of $BG$.

I thank Christian Haesemeyer, Tudor Padurariu, and a referee
for their suggestions.
This work was supported by NSF grant DMS-1303105.

\tableofcontents

\section{Notation}

A {\it variety }over a field $k$ means an integral separated scheme
of finite type over $k$. A variety $X$ over $k$ is {\it geometrically
integral }if $X_{\overline{k}}:=X\times_{\Spec(k)} \Spec(\overline{k})$
is an integral scheme 
(where $\overline{k}$ is an algebraic closure of $k$),
or equivalently if $X_E$ is integral for every extension field
$E$ of $k$ \cite[Definition IV.4.6.2]{EGAIV2}.

Let $X$ be a scheme of finite type over a field $k$.
The {\it Chow group }$CH_iX$ is the group of $i$-dimensional algebraic cycles
on $X$ modulo rational equivalence. A good reference
is Fulton \cite[Chapter 1]{Fulton}.
We write $CH_i(X;R)=CH_i(X)\otimes_{\Z}R$
for a commutative ring $R$.

For a smooth scheme $X$ over $k$, understood to be
of finite type over $k$, we write $CH^iX$
for the Chow group of codimension-$i$ cycles on $X$. For $X$ smooth over $k$,
the groups $CH^*X$ have a ring structure given by intersecting cycles
\cite[Chapter 6]{Fulton}.

\section{Birational motives}
\label{birat}

In this section, we give several equivalent characterizations
of those smooth proper varieties $X$ whose birational motive
in the sense of Rost \cite[Appendix RC]{KM}
and Kahn-Sujatha \cite[equation (2.5)]{KS}
is isomorphic to the
birational motive of a point. The statement includes Merkurjev's theorem
that the Chow group of 0-cycles on $X$ is unchanged under field extensions
if and only if the unramified cohomology of $X$ in the most general
sense is trivial
\cite{Merkurjev}.
It seems to be new that these properties are also equivalent
to all the Chow groups of $X$ being supported on a divisor. Note that
these properties are not equivalent to $CH_0$ being supported
on a divisor; for example, the product of $\P^1$ with a curve $C$
of genus at least 1 has $CH_0$ supported on a divisor, while $CH_1$
is not supported on a divisor. Also, unlike many earlier results
in this area, we work with an arbitrary coefficient ring, not just
the rational numbers.

We will use the equivalence of Theorem \ref{equiv} to give the first
counterexamples to the Chow K\"unneth property for the classifying
space of a finite group over an algebraically closed field
(section \ref{failure}).

\begin{theorem}
\label{equiv}
Let $X$ be a smooth proper variety over a field $k$,
and let $R$ be a nonzero commutative ring.
The following are equivalent.

(1) For every finitely generated field $F/k$, $CH_0(X;R)
\arrow CH_0(X_F;R)$
is surjective.

(2) For every field $F/k$, $CH_0(X;R)
\arrow CH_0(X_F;R)$
is an isomorphism, and both groups map isomorphically to $R$
by the degree.

(3) The birational motive of $X$ (in the sense of Kahn-Sujatha)
with $R$ coefficients
is isomorphic to the birational motive of a point.

(4) For every $R$-linear cycle module $M$ over $k$ (in the sense
of Rost \cite{Rost}), the homomorphism
$M(k)\arrow M(k(X))_{\nr}$ is an isomorphism. (That is, $X$
has trivial unramified cohomology in the most general sense.)

(5) There is a closed subset $S\subsetneq X$ such that
$$CH_i(X;R)/CH_i(S;R)\arrow CH_i(X_F;R)/CH_i(S_F;R)$$
is surjective for all
finitely generated fields $F/k$ and all integers $i$.
(That is, all the Chow groups 
of $X$ are constant outside a divisor.)

(6) The variety $X$ is geometrically integral,
and there is a closed subset $S\subsetneq X$ such that
$CH_i(X_F;R)/CH_i(S_F;R)=0$ for all fields $F/k$
and all $i<\dim(X)$.
\end{theorem}

For the coefficient ring $R=\Q$ which has been considered most often, there
are other equivalent statements: instead of considering all finitely
generated extension fields of $k$ in (1) or (5),
one could consider a single algebraically
closed field of infinite transcendence degree over $k$.
This gives equivalent
conditions, because $CH_*(X_F;\Q)\arrow CH_*(X_E;\Q)$ is injective
for every scheme $X$ over $F$ and every inclusion of fields $F\inj E$.
On the other hand, for the coefficient ring $R=\F_p$ 
which is of most interest for the classifying
space of a finite group, it would not be enough to consider algebraically
closed extension fields in Theorem \ref{equiv}. This follows from Suslin's
rigidity theorem: for every extension of algebraically closed fields
$F\inj E$, every $k$-scheme $X$ over $F$,
and every prime number $p$ invertible in $F$,
$CH_*(X_F;\F_p)$ maps isomorphically to $CH_*(X_E;\F_p)$
\cite[Corollary 2.3.3]{SuslinICM}.

\begin{corollary}
\label{biratcor}
Let $k$ be a perfect field which admits resolution of singularities
(for example, any field of characteristic zero).
Let $U$ be a smooth variety over $k$, not necessarily proper.
Let $R$ be a commutative ring.
If $CH^*(U;R)\arrow CH^*(U_F;R)$ is surjective for every finitely
generated field $F$ over $k$,
then the birational motive of $U$ with coefficients in $R$
is isomorphic to the birational
motive of a point.
\end{corollary}

\begin{proof}
(Corollary \ref{biratcor})
By resolution of singularities, there is a regular compactification
$X$ of $U$ over $k$, with $U=X-S$ for some closed subset $S$.
Since $k$ is perfect, the regular scheme $X$ is smooth over $k$.
Let us index Chow groups by dimension.
We use the basic exact sequence for Chow groups
\cite[Proposition 1.8]{Fulton}:

\begin{lemma}
\label{basicseq}
Let $X$ be a scheme of finite type over a field $k$.
Let $Z$ be a closed subscheme. Then the proper pushforward and flat pullback
maps fit into an exact sequence
$$CH_i(Z)\arrow CH_i(X)\arrow CH_i(X-Z)\arrow 0.$$
\end{lemma}

In the case at hand, it follows that
$CH_*(U;R)$ is isomorphic to $CH_*(X;R)/CH_*(S;R)$. So the assumption
on $U$ implies condition (5) in Theorem \ref{equiv}. The birational
motive of $U$ is (by definition) the same as that of $X$. By Theorem
\ref{equiv}, the birational motive of $X$ with coefficients in $R$
is isomorphic to the birational motive of a point.
\end{proof}

\begin{proof}
(Theorem \ref{equiv}) 
Assume condition (1). That is, $CH_0(X;R)
\arrow CH_0(X_F;R)$
is surjective for every finitely generated field $F/k$.
Let $n$ be the dimension of $X$.
The generic fiber of the diagonal $\Delta$ 
in $CH_n(X\times_k X)$ via projection to the first copy of $X$
is a zero-cycle in $CH_0X_{k(X)}$. By our assumption, the class
$[\Delta]$ in $CH_0(X_{k(X)};R)$
is the image of some zero-cycle $\alpha\in CH_0(X;R)$. For a variety $Y$
over $k$, the Chow groups of $X_{k(Y)}$ are the direct limit of the Chow
groups of $X\times_k U$, where $U$ runs over all nonempty open subsets
of $Y$. (Note that an $i$-dimensional cycle on $X\times_k U$ gives
a cycle of dimension $i-\dim(Y)$ on the generic fiber $X_{k(Y)}$.)
Therefore, we can write
$$\Delta=X\times\alpha + B$$
in $CH_n(X\times_k X;R)$, where $B$ is a cycle supported on $S\times X$
for some closed subset $S\subsetneq X$. Here we are using
the basic exact sequence for Chow groups (Lemma \ref{basicseq}).

As a correspondence, the diagonal
$\Delta$ induces the identity map from $CH_i(X;R)$
to $CH_i(X;R)$ for any $i$.
For this purpose, think of $\Delta$ as a correspondence from the first
copy of $X$ to the second. It follows that for any extension field $F$
of $k$ and any zero-cycle
$\beta$ in $CH_0(X_F;R)$, $\beta=\Delta_*(\beta)=(X\times\alpha)_*(\beta)
=\deg(\beta)\alpha$. Thus the $R$-module
$CH_0(X_F;R)$ is generated by $\alpha$
for every field $F/k$. Moreover, $\alpha$ has degree 1, and so the degree
map $\deg\colon CH_0(X_F;R)\arrow R$ is an isomorphism. We have proved
condition (2).

Condition (3) is immediate from (2). Namely, for any smooth
proper varieties $X$ and $Y$ over $k$, the set of morphisms
from the birational motive of $X$ (with $R$ coefficients) to the
birational motive of $Y$ is defined to be $CH_0(Y_{k(X)};R)$
\cite[equation (2.5)]{KS}.
So, for a point $p=\Spec(k)$, we have 
\begin{align*}
\Hom_{\bir}(p,p)&=R,\\
\Hom_{\bir}(X,p)&=CH_0(\Spec(k(X));R)=R,\\
\Hom_{\bir}(p,X)&=CH_0(X;R),\\
\Hom_{\bir}(X,X)&=CH_0(X_{k(X)};R).
\end{align*}
By (2), we know that $CH_0(X;R)$ and $CH_0(X_{k(X)};R)$ both map
isomorphically to $R$ by the degree map; so $X$ has the birational
motive of a point. It is now clear that (1), (2), and (3) are equivalent.

When $R=\Z$, Merkurjev proved that (4) is equivalent to (1) and (2)
\cite[Theorem 2.11]{Merkurjev}. The proof works with any coefficient
ring $R$. For example, to see that (3) implies (4), it suffices to
check that an element of $CH_0(Y_{k(X)};R)$ determines a pullback
map from unramified cohomology of $Y$ (with coefficients in any
$R$-linear cycle module over $k$) to unramified cohomology of $X$.

Now we show that (1) (or equivalently, (2), (3), or (4)) implies
(5) and (6). Given (1), we have a decomposition of the diagonal as above,
$$\Delta=X\times\alpha + B$$
in $CH_n(X\times_k X;R)$, where $B$ is a cycle supported on $S\times X$
for some closed subset $S\subsetneq X$. Now use the correspondence
$\Delta$ to pull cycles back from the second copy of $X$ to the first;
again, it induces the identity on Chow groups. It follows that
for any extension field $F$ of $k$ and
any cycle $\beta$ in $CH_i(X_F;R)$ with $i<n$, we have
$\beta=\Delta^*(\beta)=B^*(\beta)$, which is a cycle supported in $S$.
Thus $CH_i(S_F;R)\arrow CH_i(X_F;R)$ is surjective for all $i<n$.

To prove (6), we also have to show that $X$ is geometrically
integral. Since $X$ is smooth and proper over $k$, 
$CH_n(X_{\overline{k}};R)$ is the free $R$-module on the set
of irreducible components of $X_{\overline{k}}$, and the cycle $[X]$
is the element $(1,\ldots,1)$ in this module. But
for any irreducible component $Y$ of $X_{\overline{k}}$, with
class $(1,0,\ldots,0)$ in $CH_n(X_{\overline{k}};R)$, we have
$[Y]=\Delta^*[Y]=(X\times \alpha)^*[Y]\in R\cdot [X]=R\cdot (1,\ldots,1)$
in $CH_n(X_{\overline{k}};R)$. Since
the ring $R$ is not zero,
it follows that $X_{\overline{k}}$ is irreducible.
This proves
(6) and hence the weaker statement (5).

Finally, we prove that (5) implies (1), which will complete the proof.
This part of the argument seems to be new.
We are assuming that there is a closed subset $S\subsetneq X$
such that $CH_i(X;R)
/CH_i(S;R)\arrow CH_i(X_F;R)/CH_i(S_F;R)$ is surjective for all
finitely generated fields $F/k$ and all integers $i$.
Taking $i=n$, it follows that $X$ is geometrically integral (using
that $R$ is not zero).
As above, let
$[\Delta]$ denote the generic fiber in $CH_0(X_{k(X)};R)$ of the diagonal
$\Delta$ in $CH_n(X\times_kX;R)$. We will show
by descending induction on $j$ that for each
$0\leq j\leq n$, there is a closed subset $T_j$ of $X$
of dimension at most $j$
such that $[\Delta]$ is the image of
a zero-cycle $\alpha_j$ on $(T_j)_{k(X)}$. This is clear
for $j=n$, by taking $T_n=X$.

Suppose we have a closed subset $T_j$ and a zero-cycle $\alpha_j$
as above, for an integer $1\leq j\leq n$. Then $\alpha_j$ is the generic
fiber (with respect to the first projection) of some
$n$-dimensional cycle $A_j$ on $X\times_k T_j$. 
Let $T_{j1},\ldots,T_{jm}$ be the irreducible components
of dimension $j$ in $T_j$, and $Z$ the union of any irreducible
components of dimension less than $j$ in $T_j$. We can write
$A_j$ in $CH_n(X\times T_j;R)$ as a sum of cycles $A_{jr}$
supported on $X\times T_{jr}$, for $r=1,\ldots,m$, and a cycle $B_j$
supported on $X\times Z$. The generic fiber of $A_{jr}$
by the second projection is an $(n-j)$-cycle on $X_{k(T_{jr})}$.
By our assumption (5), this cycle is rationally equivalent to the sum of
a cycle on $S_{k(T_{jr})}$ and a cycle coming from an $(n-j)$-cycle
on $X$.
Therefore, $A_{jr}$ is equivalent to a sum of cycles
supported on $X\times Y$ for subvarieties $Y$ of dimension at most $j-1$
and cycles supported on $W\times X$ for closed subsets
$W\subsetneq X$ (using that $j>0$). This proves the inductive step:
$[\Delta]$
in $CH_0(X_{k(X)};R)$ is the image of a zero-cycle
on $(T_{j-1})_{k(X)}$ for some closed subset $T_{j-1}$ of dimension
at most $j-1$ in $X$.

At the end of the induction, we have a zero-dimensional closed subset $T_0$
of $X$ such that the class of $\Delta$ in $CH_0(X_{k(X)};R)$ is
the image of a zero-cycle $\alpha_0$ 
on $(T_0)_{k(X)}$. Here $T_0$ is a finite union of closed points,
which are isomorphic to $\Spec(E)$ for finite extension fields $E$
of $k$. Because $X$ is geometrically integral, $(\Spec(E))_{k(X)}=
\Spec(E\otimes_k k(X))$ is the spectrum of a field. So $CH_0(T_0;R)
\arrow CH_0((T_0)_{k(X)};R)$ is an isomorphism. We conclude that the class
of $\Delta$ in $CH_0(X_{k(X)};R)$ is
in the image of $CH_0(X;R)$. This gives a decomposition of the diagonal
$$\Delta=X\times \alpha + B$$
in $CH_n(X\times_k X;R)$, where $\alpha$ is a zero-cycle on $X$
and $B$ is a cycle supported on $W\times X$
for some closed subset $W\subsetneq X$. This implies statement (2),
by the same argument used to show that (1) implies (2).
Thus all the conditions
are equivalent.
\end{proof}

We now strengthen Theorem \ref{equiv} in a certain direction:
if a variety over a field $k$ has nontrivial unramified cohomology,
then its Chow groups over extension fields
of $k$ have arbitrarily large cardinality (Lemma
\ref{cardinality}). Our proof uses
the language of birational motives.
One could also give a more bare-hands argument.

\begin{lemma}
\label{chowfunctor}
Let $k$ be a field and $R$ a commutative ring.
Let $W_1$ be a variety, $W_2$ a smooth proper variety,
and $X$ a separated scheme of finite type over $k$.
For any integer $r$, there is a natural pairing
$$CH_0((W_2)_{k(W_1)};R)\otimes_R CH_r(X_{k(W_2)};R)\arrow CH_r(X_{k(W_1)};R)$$
which agrees with the obvious pullback when the 0-cycle
on $(W_2)_{k(W_1)}$ is the one 
associated to a dominant rational map $W_1\dashrightarrow W_2$.
As a result, the assignment $W\mapsto CH_r(X_{k(W)};R)$ for smooth proper
varieties $W$ over $k$ extends to
a contravariant functor on the category of birational
motives over $k$ with $R$ coefficients.
\end{lemma}

\begin{proof}
Let $M$ be an $R$-linear cycle module over $k$. The unramified cohomology
group $A^0(W;M)$ is defined in Rost \cite[section 5]{Rost} for $k$-varieties
$W$. For $W$ smooth proper over $k$, $A^0(W;M)$ is a birational invariant
of $W$, which coincides with $M(k(W))_{\nr}$ \cite[section 12]{Rost}.
Rost observed that unramified cohomology $A^0(W;M)$ 
for smooth proper varieties $W$ over $k$ is a contravariant functor
on the category
of birational motives over $k$ \cite[Theorem RC.9]{KM}.
More precisely, there is a pairing
$$CH_0((W_2)_{k(W_1)};R)\otimes_R A^0(W_2;M)\arrow A^0(W_1;M)$$
for any variety $W_1$ and any smooth proper variety $W_2$
over $k$. 

It remains to observe that for a separated scheme $X$
of finite type over $k$ and an integer $r$, there is a cycle module $M$
over $k$
with $A^0(W;M,-r)\cong CH_r(X_{k(W)};R)$ for all $k$-varieties $W$.
(The index $-r$ refers to the grading of a cycle module, as in
\cite[section 5]{Rost}.) Namely, let $M(F)$ be the $R$-linear cycle module
$A_r(X_F;K_*)$, in the notation of \cite[section 7]{Rost}. Here $F$ runs
over fields $F/k$, and $K_*$ denotes Milnor $K$-theory tensored with $R$.
Define the grading of $M(F)$ by saying
that elements of $M(F,j)$ are represented by elements of Milnor $K_{r+j}$
of function fields of $r$-dimensional subvarieties of $X_F$. Then,
by definition,
$$A^0(X;M)=\ker(A_r(Y_{k(X)},K_*)\arrow
\oplus_{x\in X^{(1)}} A_r(Y_{k(x)};K_*)).$$
The group $A_r(Y_{k(X)};K_*,R)$ is the Chow group
$CH_r(Y_{k(X)};R)$. The boundary map
takes this graded piece of $A_r(Y_{k(X)};K_*)$ to a zero group (involving
$K_{-1}$ of function fields of $r$-dimensional subvarieties of
$Y_{k(x)}$ for codimension-1 points $x$ in $X$). So
$A^0(X;M,-r)\cong CH_r(Y_{k(X)};R)$, as we want.
\end{proof}

\begin{lemma}
\label{cardinality}
Let $X$ be a separated scheme of finite type over an algebraically closed
field $k$
of characteristic zero,
$R$ a commutative ring, and $r$ an integer. Suppose that there is
a field $E/k$ such that $CH_r(X;R)\arrow CH_r(X_E;R)$ is not
surjective.
Then $CH_r(X_F;R)$ can have arbitrarily large cardinality for
fields $F/k$. In particular, there is a field $F/k$ with
$CH_r(X_F;R)$ not finitely generated as an $R$-module.
\end{lemma}

\begin{proof}
We can assume that the field $E$ is finitely generated over $k$.
Then $E$ is the function field of some variety $W$ over $k$.
Since $k$ has characteristic zero, we can assume that $W$ is smooth
and projective over $k$. Since $k$ is algebraically closed,
$W$ is geometrically integral, and so all powers
$W^n$ are varieties over $k$. Also, since $k$ is algebraically closed,
$W$ has a 0-cycle of degree 1, which we can use to give a splitting
$M_{\bir}(W)\cong M_{\bir}(k)\oplus T$ for some birational motive $T$.
So, for any natural number $n$,
$M_{\bir}(W^n)\cong (M_{\bir}(k)\oplus T)^{\otimes n}$, which contains
$M_{\bir}(k)\oplus T^{\oplus n}$ as a summand. By Lemma
\ref{chowfunctor}, it follows that for any separated $k$-scheme $X$
of finite type, we have a canonical splitting
$$CH_r(X_{k(W^n)};R)\cong CH_r(X;R)\oplus (CH_r(X_{k(W)};R)/CH_r(X;R))^{\oplus n}\oplus (\text{something}).$$

For any set $S$, let $F$ be the direct limit of the function fields
of the varieties $W^T$ over all finite subsets $T$ of $S$. Then
$CH_r(X_F;R)$ is the direct limit of the Chow groups of the varieties
$X_{k(W^T)}$. By the previous paragraph, $CH_r(X_F;R)$ contains
a direct sum of copies of $CH_r(X_{k(W)};R)/CH_r(X;R)$ indexed by the elements
of $S$. Since we assumed that $CH_r(X_{k(W)};R)/CH_r(X;R)$ is not zero,
$CH_r(X_F;R)$ can have arbitrarily large cardinality for fields $F/k$.
\end{proof}

\section{Failure of the weak Chow K\"unneth property for finite groups}
\label{failure}

We apply Theorem \ref{equiv} to give the first counterexamples to the Chow
K\"unneth property for the classifying space of a finite group $G$
over an algebraically closed field $k$, answering a question
from \cite[section 6]{Totarochow} and \cite[Chapter 17]{Totarobook}.
Namely, if the unramified cohomology of $BG$ is nontrivial,
then the weak Chow K\"unneth property fails, meaning
that there is a finitely generated field $F$ over $k$
with $CH^*BG_k\arrow CH^*BG_F$ not surjective.
Examples where $BG$ has nontrivial
unramified $H^2$ were constructed by Saltman and Bogomolov
\cite{BogomolovBrauer}.
Correcting Bogomolov's earlier statements,
Hoshi, Kang, and Kunyavskii gave examples of groups of order $p^5$
for every odd prime $p$, and groups of order $2^6$,
with nontrivial unramified $H^2$ \cite[Theorem 1.13]{HKK}. These
results are sharp for all prime numbers $p$.
Indeed, $p$-groups of order at most
$p^4$ satisfy the weak Chow K\"unneth property
\cite[Theorem 11.1, Theorem 17.4]{Totarobook}, as do all groups
of order 32 (Theorem \ref{order32}).

Chu and Kang showed that for any $p$-group $G$ of order at most $p^4$
and exponent $e$, if $k$ is a field of characteristic not $p$
which contains the $e$th roots of unity, then
$BG$ is stably rational over $k$
\cite{CK}. (Concretely, this means that the variety
$V/G$ is stably rational over $k$
for every faithful representation
$V$ of $G$ over $k$. The stable birational equivalence
class of $V/G$ for a faithful representation $V$ of a finite group $G$
is independent of the representation
$V$, by Bogomolov and Katsylo \cite{BK}.)
For 2-groups of order at most $2^5$, $BG$ is again stably rational,
by Chu, Hu, Kang, and Prokhorov \cite{CHKP}. It is striking
that $BG$ has the weak Chow K\"unneth property for $p$-groups
of order at most $p^4$, and for groups of order 32,
although there is no obvious implication between stable rationality
of $BG$ and the weak Chow K\"unneth property for $BG$. (If $BG$
can be approximated by quotients $(V-S)/G$ which are linear schemes
in the sense of section \ref{motives},
then both properties hold;
and both properties imply the triviality of unramified cohomology.)

We show in Corollary \ref{infinite} 
that for every finite group $G$ such that $BG_k$ has nontrivial
unramified cohomology with $\F_p$ coefficients,
there is an extension field $F$ of $k$
such that $CH^i(BG_F)/p$ is infinite for some $i$. This answers
another question from \cite[Chapter 18]{Totarobook}. In particular,
the ring $CH^*(BG_F)/p$ is not noetherian, and does not consist
of transferred Euler classes of representations.

We can still ask whether the abelian group $CH^iBG_F$
is finitely generated for every finite group $G$ and every integer $i$
when $F$ is an algebraically closed field. The question 
of finiteness is also interesting
for other classes of fields, such as finitely generated
fields over $\Q$ or $\F_p$. The ``motivic Bass conjecture''
\cite[Conjecture 37]{Kahn}
would imply that the Chow groups of every variety over a
finitely generated field are finitely generated; that would imply
that each group $CH^iBG_F$ is finitely generated for every affine group
scheme $G$ over a finitely generated field $F$.

Finding that the Chow K\"unneth property fails should be just
the beginning. Let $G$ be a group of order $p^5$ such that $BG$
has nontrivial unramified cohomology.
What is the Chow ring of $BG$ over an arbitrary field
(say, containing $\overline{\Q}$)?
We know that it will depend on the field.

\begin{corollary}
\label{BG}
Let $G$ be an affine group scheme of finite type
over a field $k$. Suppose that $k$ is perfect and $k$ admits
resolution of singularities (for example, $k$ could be any field
of characteristic zero). Let $p$ be a prime number which is invertible
in $k$. Suppose
that the homomorphism $H^i(k,M)\arrow H^i_{\nr}(k(V/G),M)$
of unramified cohomology
is not an isomorphism,
for some finite $\Gal(k^s/k)$-module $M$ over $\F_p$,
some generically free representation $V$ of $G$ over $k$, and some integer
$i$.
(The stable birational equivalence
class of $V/G$ for $V$ generically free is independent of 
the representation $V$,
and so this hypothesis does not depend on the choice of $V$.)
Then the weak Chow K\"unneth property with $\F_p$
coefficients fails for $BG$ over $k$, meaning
that $CH^*(BG)/p\arrow CH^*(BG_F)/p$ is not surjective for some
finitely generated field $F$ over $k$.
\end{corollary}

To relate the $p$-groups mentioned earlier to this statement,
note that those groups $G$ (of order $p^5$ for $p$ odd or order
$2^6$ for $p=2$) are shown to have nontrivial unramified Brauer group
$H^2_{\nr}(k(V/G),G_m)$, where $k$ is an algebraically closed field $k$
in which $p$ is invertible. This group is $p$-power torsion, by a transfer
argument.
By results of Grothendieck, $H^2_{\nr}(k(V/G),\mu_p)$ is isomorphic
to the $p$-torsion subgroup of $H^2_{\nr}(k(V/G),G_m)$ \cite[Proposition
4.2.3]{Colliot}. Therefore
$H^2_{\nr}(k(V/G),\mu_p)$
is also nonzero, and so Corollary \ref{BG}
applies to these groups $G$.

Explicitly, for any prime number $p\geq 5$,
here is an example of a group $G$ of order $p^5$ with 
unramified $H^2$ (over $\C$) not zero
\cite[proof of Theorem 2.3]{HKK}. In this presentation, we use the notation
$[g,h]=g^{-1}h^{-1}gh$.
\begin{multline*}
G=\langle f_1,f_2,f_3,f_4,f_5 | \, f_i^p=1 \text{ for all }i, \,
  f_5\text{ central},\\
[f_2,f_1]=f_3, \, [f_3,f_1]=f_4, \, [f_4,f_1]=[f_3,f_2]=f_5, \,
[f_4,f_2]=[f_4,f_3]=1\rangle
\end{multline*}

\begin{proof}
(Corollary \ref{BG}) 
By definition, $CH^iBG$ is isomorphic to $CH^i(V-S)/G$ for
any representation $V$ of $G$ over $k$ and any $G$-invariant
(Zariski) closed subset $S$ such that $G$ acts freely on $V-S$
with quotient a scheme and $S$ has codimension greater than $i$
in $V$ \cite[Theorem 1.1]{Totarochow}, \cite[Theorem 2.5]{Totarobook}.
By the basic exact sequence
for equivariant Chow groups, the homomorphism
$$CH^*BG=CH^*_GV\arrow CH^*_G(V-S)=CH^*(V-S)/G$$
is surjective \cite[Proposition 5]{EG}, \cite[Lemma 2.9]{Totarobook}.

Suppose that $CH^*(BG)/p\arrow CH^*(BG_F)/p$ is surjective for every
finitely generated field $F$ over $k$. Let $V$ be a
representation of $G$ with a closed subset $S\subsetneq V$
such that $G$ acts freely on $V-S$ with quotient a separated scheme
$U=(V-S)/G$.
By the previous paragraph,
applied to $G$ and $G_F$, it follows that $CH^*(U)/p\arrow
CH^*(U_F)/p$ is surjective for every finitely generated field
$F$ over $k$. By Corollary \ref{biratcor}, $U$ has the birational
motive of a point, with $\F_p$ coefficients.
It follows that the field $k(U)$ over $k$
has trivial unramified cohomology with coefficients in any
$\F_p$-linear cycle module over $k$.
Galois cohomology (with $p$ invertible
in $k$, as we assume) is an example of a cycle module.
Explicitly, for any finite
$\Gal(k_s/k)$-module $M$ killed by $p$, the assignment
$F\mapsto \oplus_{i} H^i(F, M\otimes \mu_p^{\otimes i})$
for finitely generated fields $F$ over $k$ is a cycle module over $k$
\cite[Remark 2.5]{Rost}. That completes the proof.
\end{proof}

The following corollary strengthens Corollary \ref{BG}.
We give
the first examples of finite groups $G$ and prime
numbers $p$ such that
the Chow group $CH^i(BG_F)/p$ is infinite, for some $i$
and some field $F$. Namely, we can take a group of order
$p^5$ for $p$ odd, or of order $2^6$, with nontrivial unramified
cohomology. 

\begin{corollary}
\label{infinite}
Let $G$ be a finite group, and let $p$ be a prime number.
Suppose that the unramified cohomology
$H^i_{\nr}(\overline{\Q}(V/G),\F_p)$
is not zero, for 
some generically free representation $V$ of $G$ over $\overline{\Q}$
and some $i>0$.
Then there is a field $F$ containing $\overline{\Q}$ and a positive
integer $r$ such that $CH^r(BG_F)/p$ is infinite. It follows that
the ring $CH^*(BG_F)/p$ is not noetherian.
\end{corollary}

\begin{proof}
Corollary \ref{BG} gives an extension field $E$ of $\overline{\Q}$
such that $CH^r(BG)/p\arrow CH^r(BG_E)/p$ is not surjective for some $r$.
So, for a finite-dimensional approximation $U=(V-T)/G$ to $BG$
with $T$ of codimension greater than $r$,
$CH^r(U)/p\arrow CH^r(U_E)/p$ is not surjective. By Lemma
\ref{cardinality}, there is a field $F/\overline{\Q}$ with $CH^r(U_F)/p$
infinite. Equivalently, $CH^r(BG_F)/p$ is infinite. Since
$CH^*(BG_F)/p$ is a graded $\F_p$-algebra, it follows that the ring
$CH^*(BG_F)/p$ is not noetherian.
\end{proof}

\section{The weak Chow K\"unneth property for smooth proper $k$-schemes}
\label{smoothpropersect}

In this section, we characterize the smooth proper $k$-schemes
whose Chow groups remain unchanged under arbitrary field extensions:
they are the schemes whose Chow motive is a Tate motive.
This type of result
for smooth proper $k$-schemes
has a long history, including results
by Bloch \cite[Proposition 3.12]{Kleiman},
\cite[Appendix to Lecture 1]{Blochbook}, Bloch-Srinivas \cite{BS},
Jannsen \cite[Theorem 3.5]{Jannsen},
and Kimura \cite{Kimura}. Shinder gave a convenient version
of Bloch's argument \cite{Shinder}. 
One difference from most earlier results is that we consider
Chow groups with coefficients in any commutative ring, not just
the rational numbers.

In the rest of the paper, Theorem \ref{smoothproper}
is used only to prove Corollary \ref{smoothTate}.
Nonetheless, the proof, using the diagonal cycle, helped
to suggest the proof of Theorem \ref{main} about arbitrary schemes.
Theorem \ref{equiv} is a ``birational analog'' of Theorem \ref{smoothproper};
in particular, the equivalent properties in Theorem
\ref{smoothproper} are not birationally invariant.

\begin{theorem}
\label{smoothproper}
Let $M$ be a Chow motive over a field $k$ with coefficients
in a commutative ring $R$. (For example, $M$ could be the motive $M(X)$
for a smooth proper $k$-scheme $X$.) Suppose that $M$ has
the weak Chow K\"unneth property, meaning that the morphism
$CH_*(M)\arrow CH_*(M_F)$ is a surjection of $R$-modules
for every finitely generated
field $F/k$. Then $M$ is a summand of a finite direct sum
of Tate motives $R(j)[2j]$ for integers $j$.

Conversely, suppose that a Chow motive $M$ is a summand of a finite direct
sum of Tate motives. Then
$CH_*(M)\arrow CH_*(M_F)$ is an isomorphism for every field $F/k$,
and $M$ has the Chow K\"unneth property
that $CH_*(M)\otimes_R CH_*(Y;R)\arrow
CH_*(M\otimes M^c(Y))$
is an isomorphism of $R$-modules for every separated $k$-scheme $Y$ of finite
type. Also, $CH_*(M)$ is a finitely generated projective $R$-module,
and $CH_*(M)\cong H_*(M_{\C},R)$ if there is an embedding
$k\inj \C$. Finally, $M$ has the K\"unneth property
for motivic homology in the sense that
$$CH_*(M)\otimes_R H^M_*(Y,R(*))\cong H^M_*(M\otimes M^c(Y),R(*))$$
for every separated $k$-scheme $Y$ of finite type.
\end{theorem}

The notation $M^c(Y)$ is suggested by Voevodsky's triangulated
category of motives (discussed in section \ref{motives}),
but below we say explicitly what this means.

If $R$ is a PID, then the conditions in the theorem are also equivalent
to $M$ being a finite direct sum of Tate motives (without having
to take a direct summand). For an arbitrary commutative ring $R$,
it is essential to allow direct summands.

The conclusion cannot be strengthened to say that $X$ is a linear scheme
or a rational variety. There are Barlow surfaces
over $\C$ whose Chow motive with $\Z$ coefficients
is a direct sum of Tate motives,
for example by Theorem \ref{smoothproper} and \cite[Proposition 1.9]{ACP}.
It follows that these smooth projective surfaces
have the Chow K\"unneth property,
although they are of general type and hence not rational.

Before proving Theorem \ref{smoothproper}, let us define
the category of Chow motives over $k$
with coefficients in $R$. To agree with the conventions
in Voevodsky's triangulated category of motives $DM(k;R)$ (section
\ref{motives}), we think of the basic functor $X\mapsto M(X)$ from smooth
proper $k$-schemes to Chow motives as being covariant, and we write
the motive of $\P^1_k$ as $R\oplus R(1)[2]$. Covariance is only a minor
difference from the conventions in Scholl's paper \cite{Scholl},
because the
category of Chow motives is self-dual. (The ``shift'' $[2]$ is written
in order to agree with the notation in $DM(k;R)$; it has no meaning
by itself in the category of Chow motives.) We will
only consider $DM(k;R)$ when
the exponential characteristic of $k$ is invertible in $R$;
in that case, the category of Chow motives is equivalent to a full
subcategory of $DM(k;R)$.

For smooth proper varieties $X$ and $Y$ over $k$, define
the $R$-module of correspondences of degree $r$ from $X$ to $Y$
as
$$\Corr_r(X,Y)=CH_{\dim(X)+r}(X\times_k Y;R).$$
We extend this definition to all smooth proper $k$-schemes by taking
direct sums. For smooth proper $k$-schemes $X,Y,Z$, there is a composition
of correspondences
$$\Corr_r(X,Y)\otimes_R \Corr_s(Y,Z)\arrow \Corr_{r+s}(X,Z),$$
written as $f\otimes g\mapsto gf$,
given by pulling back 
the two cycles from $X\times Y$ and $Y\times Z$
to $X\times Y\times Z$, multiplying, and pushing forward
to $X\times Z$.

A {\it Chow motive }over $k$ with coefficients in $R$, written
$(M(X)(a)[2a],p)$, consists of a smooth
proper $k$-scheme $X$, an integer $a$, and an idempotent
$p=p^2$ in $\Corr_0(X,X)$.
The morphisms of Chow motives are given by
$$\Hom((M(X)(a)[2a],p),(M(Y)(b)[2b],q))=q\Corr_{a-b}(X,Y)p
\subset \Corr_{a-b}(X,Y)$$
Composition of correspondences makes the Chow motives over $k$
into a category.
We write $M(X)$ for the motive $(M(X)(0)[0],\Delta)$, where
$\Delta$ is the diagonal in $X\times_k X$. Thus $X\mapsto M(X)$
is a covariant functor from smooth proper $k$-schemes to Chow
motives. The {\it Tate motive }$R(a)[2a]$ is $M(\Spec(k))(a)[2a]$.
Define the {\it Chow groups }of a motive $M$ by
$CH_a(M)=\Hom(R(a)[2a],M)$; then the group $CH_a(M(X))$
is isomorphic to the usual Chow group $CH_a(X;R)$ of a smooth
proper $k$-scheme $X$.

The category of Chow motives is symmetric monoidal,
with tensor product $\otimes$
such that $M(X)\otimes M(Y)\cong M(X\times_k Y)$ for smooth
proper $k$-schemes $X$ and $Y$. There is an involution $M\mapsto M^*$
on Chow motives, defined on objects by
$$(M(X)(a)[2a],p)^*=(M(X)(-n-a)[-2n-2a],p^t)$$
for $X$ of pure dimension $n$. It is immediate that the natural
morphism $M\arrow M^{**}$ is an isomorphism, and that
$$\Hom(M\otimes N,P)\cong \Hom(M,N^*\otimes P)$$
for all Chow motives $M,N,P$ \cite[section 1.1.5]{Scholl}. That is,
the category of Chow motives is a rigid additive tensor category,
with internal Hom given by $\HHom(M,N)=M^*\otimes N$.
For a field extension $F/k$, there is an obvious functor from Chow motives
over $k$ to Chow motives over $F$, taking $M(X)$ to $M(X_F)$
for smooth proper $k$-schemes $k$.

As an extension of the previous notation, for any Chow motive $M
=(M(X)(a)[2a],p)$
over $k$ and any $k$-scheme $Y$ of finite type, we define
the Chow groups $CH_*(M\otimes M^c(Y))$ as the summand
of the Chow groups $CH_*(X\times_k Y;R)$ given by $p$. (At this point,
$M^c(Y)$ has no meaning by itself. In section \ref{motives},
$M^c(Y)$ will be used to denote
the compactly supported motive of $Y$
in the triangulated category of motives $DM(k;R)$.)

\begin{proof}
(Theorem \ref{smoothproper})
Let $M$ be a Chow motive which has
the weak Chow K\"unneth property, meaning that
$CH_*M
\arrow CH_*(M_F)$ is surjective for all finitely generated fields $F$
over $k$. Then the $R$-linear map
$CH_*M\otimes_RCH_*Y\arrow CH_*(M\otimes M^c(Y))$
is surjective for every $k$-scheme $Y$ of finite type. (In this proof,
we write $CH_*(Y)$ to mean $CH_*(Y;R)=CH_*(Y)\otimes_{\Z}R$.)
To prove this,
do induction on the dimension of $Y$, using the commutative
diagram of exact sequences for any closed subscheme $S$ of $Y$:
$$\xymatrix@C-10pt@R-10pt{
CH_*M\otimes_RCH_*S  \ar[r]\ar[d] & CH_*M\otimes_RCH_*Y \ar[r]\ar[d]
& CH_*M\otimes_RCH_*(Y-S) \ar[r]\ar[d] & 0\\
CH_*(M\otimes M^c(S)) \ar[r] & CH_*(M\otimes M^c(Y)) \ar[r]
& CH_*(M\otimes M^c(Y-S)) \ar[r] & 0
}$$
Here we use that, for a $k$-variety $Y$, $CH_*(M_{k(Y)})=
\dlim CH_*(M\otimes M^c(Y-S))$, where the direct limit runs over all
closed subsets $S\subsetneq Y$. It follows
that $CH_*M\otimes_R CH_*N\arrow CH_*(M\otimes N)$ is surjective
for all Chow motives $N$.

For any Chow motives $N$ and $P$,
we have $\Hom(N,P)=\Hom(R\otimes N,P)=\Hom(R,\HHom(N,P))$.
By Lemma \ref{inthom}, the identity map on the Chow motive $M$ corresponds to 
an element $1_M\in \Hom(R,M^*\otimes M)=CH_0(M^*\otimes M)$.
(When $M$ is the motive of a smooth
proper variety $X$, $1_M$ is the class of the diagonal on $X\times X$.)

For the given motive $M$,
we showed that $CH_*M\otimes_R CH_*N\arrow CH_*(M\otimes N)$ is surjective
for all Chow motives $N$, and we apply this to $N=M^*$. So we can write
$1_M=\sum_{i=1}^r \alpha_i\otimes \beta_i$ in $CH_0(M^*\otimes M)$
for some $\alpha_1,\ldots,\alpha_r\in CH_*(M^*)$
and $\beta_1,\ldots,\beta_r\in CH_*M$. Here $\alpha_i$ is in $CH_{-b_i}(M^*)$
and $\beta_i$ is in $CH_{b_i}M$ for some integers $b_1,\ldots,b_r$.
Let $N=\oplus_{i=1}^r R(b_i)[2b_i]$. Then $(\beta_1,\ldots,\beta_r)$
can be viewed as a morphism $\beta\colon N\arrow M$, and $(\alpha_1,\ldots,
\alpha_r)$ can be viewed as a morphism $N^* \arrow M^*$,
or equivalently $\alpha\colon M\arrow N$.
The equation $1_M=\sum \alpha_i\otimes
\beta_i$ in $CH_0(M^*\otimes M)$ means that the composition 
$M\arrow N\arrow M$ is the identity.
Since idempotents split in the category of Chow motives,
it follows that $M$ is a direct summand of $N$, which is a finite
direct sum of Tate motives. One direction of the theorem is proved.

The converse statements in the theorem are clear
for a finite direct sum of Tate motives. That implies
the converse statements for any summand of a finite direct sum
of Tate motives.
\end{proof}

\section{The triangulated category of motives}
\label{motives}

This section summarizes the properties of Voevodsky's triangulated
category of motives over a field $k$, $DM(k;R)$.
Every separated scheme of finite type over $k$ (not necessarily
smooth and proper) determines an object in this category, and
Chow groups are given by morphisms from a fixed object
(a Tate motive) in this category. So $DM(k;R)$
is a natural setting for
studying Chow groups of $k$-schemes that need not be smooth and proper.

We use the triangulated category of motives for at least two purposes
in this paper. First, we need it even to state
the characterization of those schemes of finite type 
which satisfy the K\"unneth property for motivic homology groups
(Theorem \ref{main}). The corresponding characterization
for smooth proper schemes (Theorem \ref{smoothproper})
used only the more elementary
category of Chow motives. Second, we need the triangulated
category of motives in order to define the motive $M^c(BG)$
of a classifying space and to study when that motive
is mixed Tate (sections \ref{quotientsect} and \ref{TateBGsect}).

Let $k$ be a field. Thanks to recent developments in the theory
of motives, $k$ need not be assumed to be perfect or to admit
resolution of singularities. We put one restriction
on the coefficient ring $R$, as follows.
The {\it exponential characteristic }of $k$ means 1 if $k$ has characteristic
zero, or $p$ if $k$ has characteristic $p>0$. For the rest of this section,
we assume 
that the exponential characteristic of $k$ is invertible in $R$.
This assumption is used to prove the basic properties of
the compactly supported motive of a scheme $X$ over $k$, $M^c(X)$,
such as the localization triangle. (This assumption can be avoided
when we know resolution of singularities over $k$.)
This assumption on $R$ should
be understood throughout the paper when we discuss motives $M^c(X)$.

A readable introduction to Voevodsky's triangulated categories
of motives over $k$ is \cite{Voevodskytri}. Let $R$
be a commutative ring. We primarily use
the ``big'' triangulated category $DM(k;R)$ of motives with
coefficients in $R$,
which contains the direct sum of an arbitrary set
of objects. Also, the motive $R(1)$ is invertible in $DM(k;R)$,
as discussed below.
(Voevodsky originally considered the subcategory
$DM^{\text{eff}}_{-}(k)$ of ``bounded above effective motives'',
which does not have arbitrary direct sums.)
Following Cisinski and D\'eglise,
$DM(k;R)$ is defined to be the homotopy category of $G_m^{\tr}$-spectra
of (unbounded) chain complexes of Nisnevich sheaves with transfers
which are $A^1$-local
\cite[section 2.3]{RO}, \cite[Example 6.25]{CDlocal}.
For $k$ perfect, R\"ondigs and \O stv\ae r showed that
the category $DM(k;\Z)$ is equivalent to the homotopy
category of modules over the motivic Eilenberg-MacLane spectrum $H\Z$
in Morel-Voevodsky's stable homotopy category $SH(k)$ \cite[Theorem 1]{RO}.
This is an analog
of the equivalence between the derived category $D(\Z)$ of abelian groups
and the homotopy category of modules over the Eilenberg-MacLane
spectrum $H\Z$ in the category of spectra in topology
\cite[Theorem 8.9]{EKMM}.

Let $k^{\perf}$ denote the perfect closure of $k$. That is,
$k^{\perf}$ is equal to $k$ if $k$ has characteristic zero,
and $k^{\perf}$ consists of all $p^r$th roots of elements of $k$
for all $r\geq 0$ if $k$ has characteristic $p>0$. Under
our assumption that $p$ is invertible in $R$, 
Cisinski and D\'eglise proved the following convenient
result, following a suggestion
by Suslin \cite[Proposition 8.1]{CDint}.

\begin{theorem}
\label{perfect}
The pullback functor $DM(k;R)\arrow DM(k^{\perf};R)$
is an equivalence of categories.
\end{theorem}

By Theorem \ref{perfect}, most results on motives which
previously assumed that $k$ is perfect immediately
generalize to an arbitrary field $k$,
given our assumption that the exponential characteristic
of $k$ is invertible in $R$. We will
mention some examples in what follows.

By definition of a triangulated category (such as $DM(k;R)$),
every morphism $X\arrow Y$ fits into a distinguished triangle
$X\arrow Y\arrow Z\arrow X[1]$. Here $Z$ is called a {\it cone }of
the morphism $X\arrow Y$. It is unique up to isomorphism, but
not (in general) up to unique isomorphism.

There are two natural functors from
schemes to motives, which we write as $M(X)$ and $M^c(X)$.
These were defined by Voevodsky when $k$ is a perfect field
which admits resolution of singularities (as we know for $k$
of characteristic zero)
\cite[section 2.2]{Voevodskytri}. Kelly extended these
constructions to any perfect field $k$, under our assumption
that the exponential characteristic of $k$ is invertible in $R$,
using Gabber's work on alterations
\cite[Lemmas 5.5.2 and 5.5.6]{Kelly}. Finally, these constructions
now apply to any field $k$: given a separated scheme
$X$ of finite type over $k$, we have objects $M(X_{k^{\perf}})$
and $M^c(X_{k^{\perf}})$
of $DM(k^{\perf};R)$ by Kelly, hence objects $M(X)$ and $M^c(X)$ 
of $DM(k;R)$ by Theorem \ref{perfect}.

In more detail, there is a covariant
functor $X\mapsto M(X)$ from the category of separated schemes of finite
type over $k$ to $DM(k;R)$. Also, there is a covariant functor
$X\mapsto M^c(X)$ (the motive of $X$ ``with compact support'')
from the category of separated schemes of finite type
and proper morphisms to $DM(k;R)$. A flat morphism $X\arrow Y$
determines a pullback map $M^c(Y)\arrow M^c(X)$.
The motives $M(X)$ and $M^c(X)$
are isomorphic for $X$ proper over $k$.

The category $DM(k;R)$ has objects called
$R(j)$ for all integers $j$. The motives $R(j)[2j]$ are called
the {\it Tate motives}.
One interpretation of Tate motives is that $M^c(A^j_k)=R(j)[2j]$
for $j\geq 0$. More generally, for an affine bundle $Y\arrow X$
(a morphism that is locally on $X$ isomorphic to a product with
affine space $A^r$),
we have the {\it homotopy invariance }statements
that $M(Y)\cong M(X)$, whereas $M^c(Y)\cong M^c(X)(r)[2r]$.

The category $DM(k;R)$ is a tensor triangulated category, with
a symmetric monoidal product $\otimes$
\cite[Example 6.25]{CDlocal}. We have
$M(X)\otimes M(Y)=M(X\times_k Y)$ and $M^c(X)\otimes M^c(Y)=M^c(X\times_k Y)$
for $k$-schemes $X$ and $Y$
\cite[Proposition 4.1.7]{Voevodskytri}, \cite[Proposition 5.5.8]{Kelly}.
The motive $R=R(0)$
of a point is the identity object for the tensor product.
The motive $R(1)$ is invertible
in the sense that $R(a)\otimes R(b)\cong R(a+b)$ for all integers
$a$ and $b$.

The category $DM(k;R)$ has internal
$\Hom$ objects, with natural isomorphisms
$$\Hom(A\otimes B,C)\cong \Hom(A,\HHom(B,C))$$
for all motives $A,B,C$. Moreover, the internal Hom preserves
distinguished triangles in each variable, up to a sign
change in the boundary map \cite[Definition 6.6.1,
Theorem 7.1.11]{Hovey}.
(All this is part of Cisinski-D\'eglise's result that
$S\mapsto DM(S;R)$ is a ``premotivic category'' for finite-dimensional
noetherian schemes $S$ \cite[section 11.1.2]{CDtri}.) It follows
that, for any motive $B$ in $DM(k;R)$,
the functor $\cdot\otimes B$ is a left adjoint, and therefore
preserves arbitrary direct sums.

To understand the two functors, note that
the Chow groups $CH_iX$ are determined by $M^c(X)$, whereas
Chow cohomology groups $CH^iX$ for $X$ smooth over $k$ are determined
by $M(X)$. Namely,
$$CH_i(X)\otimes_{\Z}R=\Hom(R(i)[2i],M^c(X))$$
for any separated scheme $X$ of finite type over $k$, while
$$CH^i(X)\otimes_{\Z}R=\Hom(M(X),R(i)[2i])$$
for $X$ also smooth over $k$ \cite[section 2.2]{Voevodskytri}.
Voevodsky defined {\it motivic cohomology }and (Borel-Moore)
{\it motivic homology }for
any separated scheme $X$ of finite type over $k$
by 
$$H^j_M(X,R(i))=\Hom(M(X),R(i)[j])$$
and
$$H_j^M(X,R(i))=\Hom(R(i)[j],M^c(X)).$$

For a separated scheme $X$ of finite type over $k$ and a closed subscheme $Z$
of $X$, there is a distinguished triangle in $DM(k;R)$,
the {\it localization triangle}:
$$M^c(Z)\arrow M^c(X)\arrow M^c(X-Z)\arrow M^c(Z)[1].$$
(This was proved by Voevodsky when $k$ is perfect
and admits resolution of singularities \cite[section 2.2]{Voevodskytri},
by Kelly for any perfect field $k$ with our assumption on $R$
\cite[Proposition 5.5]{Kelly}, and
by Theorem \ref{perfect} for an arbitrary field $k$.)
This triangle induces a long exact sequence of motivic homology groups,
called the localization sequence.

Bloch defined higher Chow groups as the homology of an explicit
complex of algebraic cycles. Higher Chow groups are essentially
the same as motivic homology, but (by tradition) they are
numbered by codimension.
Namely, for an
equidimensional separated scheme $X$ of dimension $n$ over $k$,
$$CH^{n-j}(X,i-2j;R)\cong H_i^M(X,R(j)).$$
(For $k$ admitting
resolution of singularities and $X$ quasi-projective over $k$, this is
\cite[Proposition 4.2.9]{Voevodskytri}. Kelly modified the argument
to replace the assumption on resolution of singularities with
our assumption on $R$
\cite[Theorem 5.6.4]{Kelly}. Finally, the assumption of quasi-projectivity
was needed for Bloch's proof of the localization sequence for higher
Chow groups \cite{Bloch94},
but Levine has now proved the localization sequence
for the higher Chow groups of all schemes of finite type over a field
\cite[Theorem 0.7]{Levine}.)

Some higher Chow groups are zero by the definition, because they
consist of cycles of negative dimension or negative codimension.
It follows
that the motivic homology $H_i^M(X,R(j))$ of a separated $k$-scheme $X$
is zero unless
$i\geq 2j$ and $i\geq j$ and $j\leq \dim(X)$.

For any motive $A$ in $DM(k;R)$,
we define the {\it motivic homology }groups of $A$ to mean
the groups $H_j^M(A,R(i))=\Hom(R(i)[j],A).$
Note that what we call the motivic homology groups of a separated
$k$-scheme $X$
of finite type are the motivic homology groups of the motive $M^c(X)$,
not those of $M(X)$ (although the two motives are isomorphic
for $X$ proper over $k$).

Let ${\cal T}$ be a triangulated
category with arbitrary direct sums.
A {\it localizing subcategory }of
${\cal T}$ means a strictly full triangulated subcategory which is closed
under arbitrary direct sums. 
Following R\"ondigs and \O stv\ae r, the triangulated category
$DMT(k;R)$ of {\it mixed Tate motives }with coefficients in $R$
is the smallest localizing subcategory of $DM(k;R)$
that contains $R(j)$ for all integers $j$ \cite{RO}.
Because the tensor product
$\otimes$ on $DM(k;R)$ is compatible with distinguished triangles
and with arbitrary direct sums, the tensor product of two mixed Tate motives
is a mixed Tate motive.

The category of mixed Tate motives is analogous
to the category of {\it cellular }spectra in the stable
homotopy category $SH(k)$ studied
by Voevodsky \cite{Voevodskyopen}
and Dugger-Isaksen \cite{DI}. (Actually, Voevodsky says ``$T$-cellular''
and Dugger and Isaksen say ``stably cellular''.) Namely,
let $T$ be the suspension spectrum of the pointed
$k$-space $(\P^1_k,\text{point})$; the triangulated category
of cellular spectra is defined as the smallest localizing
subcategory of $SH(k)$ that contains $T^j$ for all integers
$j$.

As with motives, there
are two natural functors from separated $k$-schemes $X$ of finite type
to $SH(k)$: the usual functor (which we write as $X\mapsto S(X)$
or $X\mapsto \Sigma^{\infty}_T X^{+}$)
and a compactly supported version, $X\mapsto S^c(X)$.
Explicitly, for any compactification $\overline{X}$ of a $k$-scheme
$X$, $S^c(X)$ is the spectrum associated to the pointed
$k$-space $\overline{X}/(\overline{X}-X)$. There is a functor from
$SH(k)$ to $DM(k;R)$, which one can view as smashing with the Eilenberg-MacLane
spectrum $HR$, and this takes $S(X)$ to $M(X)$ and $S^c(X)$
to $M^c(X)$. In particular, the spectrum $T$ goes to the motive $R(1)[2]$.

In a triangulated category with arbitrary direct sums, every idempotent
splits \cite[Proposition 3.2]{BN}.
Applying this to the category of mixed Tate motives,
it follows that every summand of a mixed Tate motive in $DM(k;R)$
is a mixed Tate motive.

Let ${\cal T}$ be a triangulated
category with arbitrary direct sums. An object $X$
of ${\cal T}$ is called {\it compact }if $\Hom(X,\cdot)$ commutes with
arbitrary direct sums. The objects $M(X)(a)[b]$ and $M^c(X)(a)[b]$
are compact in $DM(k;R)$ for every separated $k$-scheme $X$ of finite type
\cite[Lemmas 5.5.2 and 5.5.6]{Kelly}.
A set ${\cal P}$ of objects {\it generates }${\cal T}$ if every object
$Y$ of ${\cal T}$ such that $\Hom(P[a],Y)=0$ for all objects $P$ in ${\cal P}$
and all integers $a$ is zero. A triangulated category ${\cal T}$
is {\it compactly generated }if it has arbitrary direct sums and it
is generated by a set of compact objects.

The following result by Neeman helps to understand the notion
of generators for a triangulated category
\cite[Theorem 2.1]{Neeman}. 

\begin{lemma}
\label{gens}
Let ${\cal T}$ be a triangulated category with arbitrary direct sums,
and let ${\cal P}$ be a set of compact objects. The following
are equivalent:

(1) The smallest localizing subcategory of ${\cal T}$
that contains ${\cal P}$ is equal to
${\cal T}$.

(2) The set ${\cal P}$ generates ${\cal T}$. That is, any object $X$
in ${\cal T}$ with $\Hom(P[a],X)=0$ for all $P$ in ${\cal P}$
and $a\in\Z$ must be zero.
\end{lemma}

\begin{corollary}
\label{mixed}
A mixed Tate motive with zero motivic homology must be zero.
\end{corollary}

\begin{proof}
The objects $R(a)$ for integers $a$ generate the category
$DMT(k;R)$, by Lemma \ref{gens}. Since $H_j^M(A,R(i))=\Hom(R(i)[j],A)$
for a motive $A$, the corollary is proved.
\end{proof}

\begin{lemma}
\label{smoothgens}
Let $k$ be a field. Then
the category $DM(k;R)$ is compactly generated, with a set
of generators given by
the compact objects $M(X)(a)$ for $X$ smooth projective over $k$
and $a$ an integer.
\end{lemma}

\begin{proof}
This was proved by Voevodsky when $k$ is perfect and admits
resolution of singularities \cite[Corollary 3.5.5]{Voevodskytri}.
Given our assumption that the exponential characteristic of $k$
is invertible in $R$, Kelly generalized this result
to any perfect field $k$
\cite[Proposition 5.5.3]{Kelly}. The generalization
to an arbitrary field $k$ follows from Theorem \ref{perfect}.
\end{proof}

A reassuring fact is that if the motive $M^c(X)$ in $DM(k;R)$
of a separated $k$-scheme $X$ of finite type
is mixed Tate, then it is a summand of an object of
the smallest strictly full triangulated subcategory
of $DM(k;R)$ that contains $R(j)$ for all integers $j$. In other words,
$M^c(X)$ can be described by a {\it finite }diagram of objects $R(j)$.
This follows from a general result about triangulated categories.
Define a {\it thick }subcategory of a triangulated category
to be a strictly full triangulated subcategory
that is closed under direct summands.
Let ${\cal T}$ be a compactly generated triangulated category,
and let ${\cal P}$ be a set of compact generators. (We have in mind
the category of mixed Tate motives, generated by the objects $R(j)$
for integers $j$.) Then Neeman showed that any compact object in ${\cal T}$
belongs to the smallest thick
subcategory of ${\cal T}$ that contains ${\cal P}$
\cite[Theorem 2.1]{Neeman}.

The category $DM_{\gm}(k;R)$ of geometric motives is defined as
the smallest thick subcategory of $DM(k;R)$ that
contains $M(X)(a)$ for all smooth
separated schemes $X$ of finite type over $k$ and all integers $a$.
In fact,
it suffices to use $M(X)(a)$ for smooth projective varieties
$X$ over $k$ and all integers $a$, by Lemma \ref{smoothgens}.
Another application of Neeman's theorem gives that $DM_{\gm}(k;R)$ is
the subcategory of all compact objects in $DM(k;R)$.

A {\it linear scheme }over a field $k$ is defined inductively: affine space
$A^n_k$ is a linear scheme for any $n\geq 0$; for any scheme $X$
of finite type over $k$ with a closed subscheme $Z$, if $Z$ and $X-Z$
are linear schemes, then $X$ is a linear scheme; and if $X$ and $Z$
are linear schemes, then $X-Z$ is a linear scheme.
(A slightly narrower class of linear schemes was studied
in \cite{Totarolinear}.)
Some examples of linear schemes
are all toric varieties, not necessarily smooth or compact, the discriminant
hypersurface and its complement, and many quotients of affine space by finite
group actions. Linear schemes can have torsion in their Chow groups
and homology groups,
and they can have nonzero rational homology in odd degrees.
(To talk about rational homology, assume that the base field
is the complex numbers.)

From the localization triangle,
a straightforward induction shows that for any linear scheme $X$ over $k$,
the compactly supported motive $M^c(X)$ with any coefficient ring $R$
is a mixed Tate motive.
Likewise, for any linear scheme $X$, the spectrum $S^c(X)$ is cellular
in $SH(k)$.
(Dugger and Isaksen asked whether the spectrum $S(X)$ is cellular
for linear schemes $X$, and proved this in some examples
\cite[section 1.1]{DI}.
Arguably, the more natural spectrum associated to a linear scheme
$X$ is $S^c(X)$, which is clearly cellular. For $X$ proper over $k$,
$S(X)$ and $S^c(X)$ are isomorphic.)

Let $X$ and $Y$ be smooth proper varieties over $k$. Then
the set of morphisms from $M(X)$ to $M(Y)$ in $DM(k;R)$
is the Chow group $CH_{\dim(X)}(X\times_k Y;R)$
\cite[section 2.2]{Voevodskytri}. Composition of morphisms
$M(X)\arrow M(Y)\arrow M(Z)$
is given by the composition of correspondences.
As a result, the smallest strictly full subcategory of $DM(k;R)$
that is closed under direct summands and contains $M(X)(a)[2a]$
for all smooth proper schemes
$X$ over $k$ and all integers $a$ is equivalent
to the category of Chow motives over $k$ with coefficients in $R$,
as defined in section \ref{smoothpropersect}.

We define $N^*=\HHom(N,R)$. A version of Poincar\'e duality says that
$M^c(X)\cong M(X)^*(n)[2n]$ for $X$ smooth
of pure dimension $n$ over $k$ \cite[Theorem 5.5.14]{Kelly}.
The internal Hom of motives has a simple description
for compact objects, as follows.

\begin{lemma}
\label{inthom}
Let $M$ be an object of $DM_{\gm}(k;R)$, 
for example the motive $M^c(X)(a)[b]$ for a scheme $X$ of finite type over $k$
and $a,b\in \Z$. Let $N$ be any object of $DM(k;R)$.
Then
the morphism $M^*\otimes N\arrow \HHom(M,N)$ is an isomorphism.

Also, for $M$ in $DM_{\gm}(k;R)$, the natural map $M\arrow M^{**}$
is an isomorphism.
\end{lemma}

\begin{proof}
At first, let $M^*$ denote the object $\HHom_{\gm}(M,R)$ in the subcategory
$DM_{\gm}(k;R)$ of compact objects.
Then Voevodsky and Kelly prove that $M\arrow M^{**}$ is an isomorphism for 
$M$ compact, and also that $M^*\otimes N\arrow \HHom_{\gm}(M,N)$
is an isomorphism for $M$ and $N$ compact
\cite[Theorem 4.3.7]{Voevodskytri}, \cite[Theorem 5.5.14]{Kelly}.
That is, the map
$$\Hom(A,B^*\otimes C)\arrow \Hom(A\otimes B,C)$$
associated to $B^*\otimes B\arrow R$ is a bijection
for all compact objects $A,B,C$.

For $A$ and $B$ compact, the map of $\Hom$ sets above turns arbitrary direct
sums of motives $C$ into direct sums, and fits into long
exact sequences for any distinguished triangle of objects $C$.
By Lemmas \ref{gens} and \ref{smoothgens},
it follows that the map is an isomorphism for $A$ and $B$ compact and $C$
arbitrary. For $B$ compact and $C$ arbitrary, both $\Hom$ sets turn
arbitrary direct sums of motives $A$ into products, and they fit
into long exact sequences for any distinguished triangle of objects $A$.
Therefore the map is an isomorphism for $B$ compact and $A$ and $C$ any
motives. That is, the internal Hom in $DM(k;R)$ has
$\HHom(B,C)\cong B^*\otimes C$ for $B$ compact and $C$ arbitrary.
In particular, taking $C=R$, we see that the object $B^*$
(which we defined as $\HHom_{\gm}(B,R)$ in $DM_{\gm}(k;R)$)
is isomorphic to $\HHom(B,R)$ in $DM(k;R)$.
\end{proof}

\section{A K\"unneth spectral sequence for motivic homology}

Dugger and Isaksen proved
the following K\"unneth spectral sequence, which describes
the motivic homology of the tensor product of a mixed
Tate motive with any motive \cite[Proposition 7.10]{DI}.
Their result applies to modules over any ring spectrum
in the stable homotopy category over a field $k$; the case of the
Eilenberg-MacLane spectrum $HR$ in $SH(k)$ gives the result here,
by the identification between the homotopy category of $HR$-module spectra
and $DM(k;R)$ \cite[Theorem 1]{RO}. (It is also straightforward
to translate Dugger and Isaksen's
proof to work directly in $DM(k;R)$.)
In the case
of the product of a linear scheme with any scheme over a field,
this spectral sequence was constructed by Joshua
\cite{Joshua}. 

\begin{theorem}
\label{ss}
Let $k$ be a field.
Let $R$ be a commutative ring.
Let $X$ be a mixed Tate motive in $DM(k;R)$ and $Y$ any motive in $DM(k;R)$.
For each integer $j$, there is a convergent spectral sequence
$$E_2^{pq}=\Tor^{H_*(k,R(*))}_{-p,-q,j}(H_*(X,R(*)),H_*(Y,R(*)))\imp
H_{-p-q}(X\otimes Y,R(j)).$$
This spectral sequence
is concentrated in the left half-plane (columns $\leq 0$).
\end{theorem}

By the discussion after Theorem \ref{main}, one can define
a spectral sequence with the $E_2$ term above for any motives
$X$ and $Y$ in $DM(k;R)$. It does not always converge
to the motivic homology of $X\otimes Y$.

We use cohomological numbering, which means
that the differential $d_r$ has bidegree $(r,1-r)$ for all $r$.

For bigraded modules $M$ and $N$ over a bigraded ring $S$,
$\Tor^S_{a,i,j}(M,N)$ denotes the $(i,j)$th bigraded piece
of $\Tor^S_a(M,N)$. For this purpose, the group
$H_i^M(X,R(j))$ has bigrading $(i,j)$.

Here $H_i(k,R(j))\cong H^{-i}(k,R(-j))$, and so the ring $H_*(k,R(*))$
is better known as the motivic cohomology ring of $k$ with coefficients
in $R$. For example, $H_{-1}(k,\Z(-1))$ is isomorphic to $k^*$.
More generally, $\oplus_{j\geq 0}H_{-j}(k,\Z(-j))$ is the Milnor
$K$-theory ring, that is, the quotient of the tensor algebra generated
by the abelian group $k^*$ by the relation
$\{a,1-a\}=0$ for each $a\in k-\{0,1\}$
\cite{NS, TotaroMilnor}.

If $X$ and $Y$ are $k$-schemes, viewed as the motives $M^c(X)$
and $M^c(Y)$,
then the spectral sequence with $R(j)$ coefficients
is concentrated in columns $\leq 0$ and rows $\leq -2j$.
If we write $H_*(X)$ for the bigraded group $H_*(X,R(*))$, the $E_2$
term looks like:
$$\xymatrix@C-10pt@R-10pt{
0 & 0 & 0 & 0\\
[\Tor_2^{H_*k}(H_*X,H_*Y)]_{2j,j} \ar[rrd] & [\Tor_1^{H_*k}
(H_*X,H_*Y)]_{2j,j} & [H_*X\otimes_{H_*k}H_*Y]_{2j,j} & 0\\
[\Tor_2^{H_*k}(H_*X,H_*Y)]_{2j+1,j} & [\Tor_1^{H_*k}
(H_*X,H_*Y)]_{2j+1,j} & [H_*X\otimes_{H_*k}H_*Y]_{2j+1,j} & 0
}$$
(Indeed,
for a $k$-scheme $X$,
the group $H_a(X,R(b))$ is zero unless $a\geq 2b$, as mentioned
in section \ref{motives}. Since this applies to $X$, $Y$, and 
$\Spec(k)$, the $E_2$ term for the spectral sequence
with $R(j)$ coefficients is concentrated in rows $\leq -2j$.)
So there are no differentials into or out of the upper right
group, $E_2^{0,-2j}$. We deduce
that 
$$CH_*(X\times_k Y;R)\cong
CH_*(X;R)\otimes_RCH_*(Y;R)$$
if $X$ is a $k$-scheme with $M^c(X)$ a mixed Tate motive in $DM(k;R)$
and $Y$ is any separated $k$-scheme of finite type. I proved
this in the special case where
$X$ is a linear scheme over $k$ \cite{Totarolinear}, which helped
to inspire Joshua's result.

\section{The motivic K\"unneth property}
\label{kunnethsect}

In this section, we prove that a separated scheme $X$ 
of finite type over a field $k$
satisfies the motivic K\"unneth property
if and only if the motive $M^c(X)$ is a mixed Tate motive.
Given the machinery we have developed, the proof is short.

The motivic K\"unneth property means that the spectral sequence
described in Theorem \ref{ss} converges to the motivic
homology of $X\times_k Y$ for every separated $k$-scheme $Y$
of finite type. (We recall that motivic homology groups are also
called higher Chow groups.)
There is a neater formulation of the K\"unneth
property in the language of Bousfield localization, to be explained now.

The inclusion of mixed Tate motives $DMT(k;R)$ into the category
$DM(k;R)$ of all motives has a right adjoint $DM(k;R)
\arrow DMT(k;R)$,
which we write as $X\mapsto C(X)$.
It associates to any motive a mixed Tate motive with the same motivic
homology groups. For $X$ a compact object (a geometric motive),
$C(X)$ need not be a compact object. So this construction
shows the convenience of ``big'' categories
of motives. The construction is a general application of Bousfield
localization, as developed by Neeman for triangulated categories.

Namely, let ${\cal T}$ be a triangulated
category with arbitrary direct sums. 
Let ${\cal P}$ be a set of compact
objects in ${\cal T}$. Recall from section \ref{motives} that a localizing
subcategory of ${\cal T}$ means a full triangulated subcategory
which is closed under arbitrary direct sums.
Let ${\cal S}$ be the smallest localizing category
that contains ${\cal P}$. Then the inclusion ${\cal S}\arrow {\cal T}$
has a right adjoint $C\colon {\cal T}\arrow {\cal S}$ known as
{\it colocalization }with respect to ${\cal P}$
\cite[Theorem 4.1]{Neeman}. By adjointness,
there is a canonical morphism $C(X)\arrow X$, and this morphism induces
a bijection $\Hom(P[j],C(X))\arrow \Hom(P[j],X)$ for all objects
$P$ in ${\cal P}$ and all integers $j$. (The {\it localization }of an
object $X$ with respect to ${\cal P}$ means a cone $X/C(X)$, which in this
case is defined up to a unique isomorphism.)

The functor $DM(k;R)\arrow DMT(k;R)$, $X\mapsto C(X)$,
mentioned above
is the colocalization with respect to the compact objects $R(j)$
for $j\in\Z$. The construction implies that $C(X)$ is a mixed Tate motive
with a morphism $C(X)\arrow X$ that induces isomorphisms on
motivic homology groups. (That is, $\Hom(R(a)[b],
\linebreak[0]
C(X))\arrow
\Hom(R(a)[b],
\linebreak[0]
X)$ is an isomorphism for all integers $a$ and $b$.)
Moreover, $C(X)$ is determined up to a unique isomorphism by this
property.

As in any triangulated category with arbitrary
direct sums, the {\it homotopy colimit }$X_{\infty}=\hocolim (X_0\arrow X_1
\arrow \cdots)$ is defined as a cone of the morphism
$$1-s\colon \oplus_{i\geq 0}X_i\arrow \oplus_{i\geq 0}X_i,$$
where $s$ is the given map from each $X_i$ to $X_{i+1}$
\cite{BN}.

Here is an explicit construction of the colocalization $C(X)$,
modeled on Dugger and Isaksen's analogous construction in the
stable homotopy category over $k$ \cite[Proposition 7.3]{DI}.
(They were imitating the usual construction
of a {\it cellular approximation }to any topological space.)
Choose a set of generators for all the motivic homology groups
$H_b(X,R(a))$ with $a,b\in \Z$.
Let $C_0$ be a direct sum of one motive $R(a)[b]$ for each
generator; so we have a morphism $C_0\arrow X$ that induces
a surjection on motivic homology groups. Next, choose a set
of generators for the kernel of $H_*(C_0,R(*))\arrow 
H_*(X,R(*))$, let $S_1$ be the corresponding direct sum
of motives $R(a)[b]$, and let $C_1$ be a cone of the morphism
$S_1\arrow C_0$. Then we have a morphism $C_0\arrow C_1$, and we 
can choose an extension of the morphism $C_0\arrow X$ to $C_1\arrow X$.
Repeating the process, we get a sequence of mixed Tate motives
$$C_0\arrow C_1\arrow\cdots $$
with a compatible sequence of morphisms $C_i\arrow X$. 
These extend to a morphism from the homotopy colimit,
$\hocolim_j\, C_j\arrow X$. This homotopy colimit
is a mixed Tate motive, and the morphism induces an isomorphism
on motivic homology groups. So the colocalization $C(X)$ is isomorphic
to $\hocolim_j \, C_j$.

By Corollary \ref{mixed}, any mixed Tate motive with zero
motivic homology groups is zero. This is not true for motives
in general. In fact, for any motive $X$, the cone of $C(X)\arrow X$ has
motivic homology groups equal to zero, and it is zero if and only if 
$X$ is a mixed Tate motive.

\begin{lemma}
\label{colocsum}
The colocalization functor $X\mapsto C(X)$ from $DM(k;R)$
to $DMT(k;R)$ preserves arbitrary direct sums and arbitrary
products. 
\end{lemma}

\begin{proof}
Because the category $DMT(k;R)$ is compactly generated, it has
arbitrary products \cite[Proposition 8.4.6]{Neemanbook}.
(Beware that the inclusion $DMT(k;R)\arrow DM(k;R)$ preserves arbitrary
direct sums,
but need not preserve arbitrary products.) Because the functor
$X\mapsto C(X)$ from $DM(k;R)$ to $DMT(k;R)$
is a right adjoint, it preserves arbitrary products.
Because the functor $X\mapsto C(X)$ is colocalization with respect
to a set of compact objects in $DM(k;R)$ (namely $R(j)$ for integers
$j$), it also preserves arbitrary direct sums
\cite[Theorem 5.1]{Neeman}.
\end{proof}

For any motives
$X$ and $Y$ in $DM(k;R)$, there is a canonical morphism
$$C(X)\otimes C(Y)\arrow C(X\otimes Y),$$
generally not an isomorphism.
Indeed, tensoring the morphisms $C(X)\arrow X$ and $C(Y)\arrow Y$
gives a morphism $C(X)\otimes C(Y)\arrow X\otimes Y$. Since
$C(X)\otimes C(Y)$ is a mixed Tate motive, this morphism factors
uniquely through $C(X\otimes Y)$, as we want.

\begin{theorem}
\label{main}
Let $k$ be a field.
Let $R$ be a commutative ring.
Let $X$ be an object
of the category $DM(k;R)$ of motives (for example,
$X$ could be the motive $M^c(W)$ for a separated $k$-scheme $W$
of finite type,
if the exponential characteristic of $k$ is invertible in $R$).
The following are equivalent.

(1) $X$ is a mixed Tate motive.

(2) $X$ satisfies the motivic K\"unneth property, meaning
that the morphism
$$C(X)\otimes C(M(Y))\arrow C(X\otimes M(Y))$$
of mixed Tate motives is an isomorphism for every
smooth projective variety $Y$ over $k$.

(3) $X$ satisfies the apparently stronger property that
$$C(X)\otimes C(Y)\arrow C(X\otimes Y)$$
is an isomorphism for every motive $Y$ in $DM(k;R)$.

If $X$ belongs to the subcategory $DM_{\gm}(k;R)$ of geometric motives,
for example if $X=M^c(B)$ for some separated $k$-scheme $B$ of finite type,
then (1)--(3) are also equivalent to:

(4) $X$ is a ``small'' mixed Tate motive, meaning that $X$ belongs
to the smallest thick subcategory 
of $DM(k;R)$ that contains $R(j)$ for all integers $j$.
\end{theorem}

Let us explain why properties (2) and (3) deserve to be called K\"unneth
properties of $X$. Since $C(X)\otimes C(Y)$ and $C(X\otimes Y)$
are both mixed Tate motives, the morphism $C(X)\otimes C(Y)
\arrow C(X\otimes Y)$ is an isomorphism if and only if it induces
an isomorphism on motivic homology groups, by Corollary
\ref{mixed}. The motivic homology
groups of $C(X\otimes Y)$ are simply the motivic homology groups
of $X\otimes Y$. The motivic homology groups of $C(X)\otimes C(Y)$
are the ``output'' of the spectral sequence of Theorem
\ref{ss}, with $E_2$ term 
$$\Tor_*^{H_*(k,*)}(H_*(C(X),R(*)),
H_*(C(Y),R(*)))=\Tor_*^{H_*(k,*)}(H_*(X,R(*)),
H_*(Y,R(*))).$$
So property (3) is saying that this K\"unneth
spectral sequence converges to the motivic homology of $X\otimes Y$.

\begin{proof}
The K\"unneth property (2) is preserved under
arbitrary direct sums of motives $X$, since the tensor product $\otimes$
and the functor $X\mapsto C(X)$ (by Lemma
\ref{colocsum}) preserve arbitrary direct sums.
Also, if it holds for two of the three
motives in a distinguished triangle, then it holds for the third.
Finally, the motives $R(j)$ have the K\"unneth property. It follows
that every mixed Tate motive in $DM(k;R)$ has the K\"unneth
property. That is, (1) implies (2).

Next, let $X$ be a motive in $DM(k;R)$ with the K\"unneth property
(2) with respect to smooth projective varieties over $k$. The
statement that the morphism
$$C(X)\otimes C(Y)\arrow C(X\otimes Y)$$
is an isomorphism is preserved under arbitrary direct sums of motives $Y$.
Also, if it holds for two motives $Y$ in a distinguished triangle,
then it holds for the third. By Lemma \ref{smoothgens},
$X$ satisfies the K\"unneth
property (3) with respect to all motives $Y$.

We now show that (3) implies (1). 
As above, the ``cellular approximation''
$C(X)$ is the unique mixed Tate motive with a morphism
$C(X)\arrow X$ that induces an isomorphism on motivic homology groups.
Since $C(X)$ is a mixed Tate motive,
it has the K\"unneth property. Let $X_2$ be a cone of the morphism
$C(X)\arrow X$. It suffices to show that $X_2=0$.

The motivic homology groups of $X_2$ are equal to zero. Also,
$X_2$ satisfies the K\"unneth property. So the motivic homology
of $X_2\otimes Y$ is zero for every motive $Y$ in $DM(k;R)$. In particular,
for all smooth projective varieties $Y$ over $k$ and all integers
$a$ and $b$,
the motivic homology group $\Hom(R,
\linebreak[0]
X_2\otimes (M(Y)(a)[b])^*)$ is zero.
By Lemma \ref{inthom}, it follows that $\Hom(M(Y)(a)[b],
\linebreak[0]
X_2)=0$
for all smooth projective varieties $Y$ over $k$ and all integers
$a$ and $b$.
By Lemma \ref{smoothgens},
it follows that $X_2=0$. We have shown that (3) implies (1).

Finally, if $X$ belongs to the subcategory $DM_{\gm}(k;R)$
of geometric motives, then we showed after Lemma \ref{smoothgens}
that (1) and (4) are equivalent.
\end{proof}

The following consequence is not surprising, but it seems worth
mentioning. Dugger and Isaksen mentioned that it is not immediately
clear how to show that a given object in the stable homotopy category,
$SH(k)$, for example an elliptic curve over $k$, is not cellular 
\cite[section 1.2]{DI}.
The functor $SH(k)\arrow DM(k;R)$ takes cellular objects to mixed
Tate motives. The following result describes which smooth projective
varieties have motives which are mixed Tate motives. As a very special
case, we see that elliptic curves are not mixed Tate motives (for
any nonzero coefficient ring), and so elliptic curves
are not cellular in $SH(k)$.

\begin{corollary}
\label{smoothTate}
Let $X$ be a smooth proper scheme over a field $k$. Let $R$
be a commutative ring such that the exponential characteristic of $k$
is invertible in $R$. If the motive $M(X)$ in $DM(k;R)$ is a mixed
Tate motive, then the Chow motive of $X$ with coefficients in $R$
is a summand
of a finite direct sum of Tate motives $R(a)[2a]$. So, for example,
$CH_*(X)\otimes_{\Z}R\arrow H_*(X_{\C},R)$ is an isomorphism if
there is an embedding $k\inj \C$. In particular, $H_*(X_{\C},R)$
is concentrated in even degrees.
\end{corollary}

\begin{proof}
By Theorem \ref{main}, $X$ satisfies the K\"unneth property
for motivic homology groups with coefficients in $R$.
By the discussion of the K\"unneth
spectral sequence after Theorem \ref{ss}, it follows that
$X$ has the Chow K\"unneth property: the homomorphism
$$CH_*(X;R)\otimes_RCH_*(Y;R)\arrow CH_*(X\times_k Y;R)$$
is an isomorphism for every 
separated $k$-scheme $Y$ of finite type. By Theorem \ref{smoothproper},
the Chow motive of $X$ with coefficients in $R$
is a summand
of a finite direct sum of Tate motives. The theorem includes several
consequences of that property, for example that
$CH_*(X;R)\arrow H_*(X_{\C},R)$ is an isomorphism if
there is an embedding $k\inj \C$.
\end{proof}

\section{The motive of a quotient stack}
\label{quotientsect}

Edidin and Graham defined the motivic homology of a quotient
stack \cite[sections 2.7 and 5.3]{EG}.
In this section, we define the compactly supported motive
of a quotient stack, in such a way that we recover the same
motivic homology groups. One benefit of defining
the motive of a quotient stack is that it makes sense to ask
whether a given stack, such as $BG$
for an affine group scheme $G$,
is mixed Tate, meaning that the motive $M^c(BG)$ is mixed Tate.

The motive $M(BG)$ (not compactly supported) in $DM(k;R)$
was already defined,
in effect, by Morel and Voevodsky \cite[section 4.2]{MV}.
Its motivic cohomology
is the motivic cohomology of $BG$. We need to define $M^c(BG)$
because that is the motive relevant to the motivic homology
of $BG\times X$ for separated schemes $X$ of finite type over $k$.
To see the difference between the two motives, write $G_m$ for
the multiplicative group over $k$. Then
$M(BG_m)$ is the homotopy
colimit of the motives $M(\P^j)$, and so $M(BG_m)$ is isomorphic
to $\oplus_{j\geq 0} \Z(j)[2j]$ in $DM(k;\Z)$.
By contrast, $M^c(BG_m)$ is the homotopy
limit of the motives $M(\P^{j-1})(-j)[-2j]$ by the definition below,
and so $M^c(BG_m)$ is isomorphic to $\prod_{j\leq -1}\Z(j)[2j]$
in $DM(k;\Z)$. This product is isomorphic
to the direct sum
$\oplus_{j\leq -1}\Z(j)[2j]$ by Lemma \ref{bgm}, from which we see
that $M^c(BG_m)$ is mixed Tate.

Another possible name for the mixed Tate property of $BG$ would be
the {\it motivic K\"unneth property}. Indeed,
by Theorem \ref{main}, $BG$ is mixed Tate
if and only if $BG$ has the motivic K\"unneth property
in the sense that the K\"unneth spectral sequence
$$E_2^{pq}=\Tor^{H_*(k,*)}_{-p,-q,j}(H_*(BG,R(*)),H_*(Y,R(*)))\imp
H_{-p-q}(BG\times_k Y,R(j))$$
converges to the groups on the right
for every separated $k$-scheme $Y$ of finite type.

Before defining the compactly supported motive of a quotient stack,
we recall the definition of homotopy limits.
Let
$$\cdots\arrow X_2 \arrow X_1$$
be a sequence of morphisms in the category $DM(k;R)$ of motives.
Since $DM(k;R)$ is compactly generated, arbitrary
products exist in $DM(k;R)$
\cite[Proposition 8.4.6]{Neemanbook}. Dualizing B\"okstedt
and Neeman's definition of homotopy colimits,
the {\it homotopy limit }$\holim_j X_j$ in $DM(k;R)$ is defined
as the fiber of the morphism $f\colon
\prod X_j\arrow\prod X_j$ given by the identity
minus the shift map \cite{BN}.
(In other words, the homotopy limit is $\cone(f)[-1]$;
so it is well-defined up to isomorphism, but not necessarily up to a unique
isomorphism.)

Define a {\it quotient stack }over a field $k$ to be an algebraic
stack over $k$ which is the quotient stack of some quasi-projective
scheme $Y$ over $k$ by an action of an affine group scheme $G$ of finite
type over $k$ such that there is a $G$-equivariant ample line bundle
on $Y$. (A short introduction to quotient stacks
is \cite[Tag 04UV]{Stacks}.
It would be more natural to allow quotients of algebraic spaces
by affine group schemes,
but this definition of quotient stacks is sufficient for our applications.)
For $Y$ quasi-projective over $k$, the assumption that there is
a $G$-equivariant ample line bundle is automatic when $G$ is finite,
or when $G$ is smooth over $k$ and $Y$ is normal, by Sumihiro's
equivariant completion theorem \cite{Sumihiro}
and \cite[Corollary 1.6]{GIT}. 
For example, the stack $BG$ means the quotient stack $\Spec(k)/G$.

The {\it dimension }of a locally Noetherian stack in defined in such a way
that a quotient stack $X=A/G$
has dimension $\dim(A)-\dim(G)$
\cite[Tag 0AFL]{Stacks}. 
For example, $BG$ is a smooth stack of dimension $-\dim(G)$ over $k$.

We can now define the compactly supported motive of a quotient stack,
with coefficients in a given commutative ring $R$.
Let $k$ be a field. Let $R$ be a commutative ring in which
the exponential characteristic of $k$ is invertible.
Let $X$ be a quotient stack over $k$.
Let $\cdots \surj V_2 \surj V_1$ be a sequence of surjections
of vector bundles
over $X$. 
Write $n_i$ for the rank of the bundle $V_i$.
Think of the total space of $V_i$ as a stack over $k$.
For each $i$, let $S_i$ be a closed substack
of $V_i$ such that $V_i-S_i$ is a separated scheme
and $S_{i+1}$ is contained in the inverse image of $S_i$ under the morphism
$f_i\colon V_{i+1}\surj V_i$ for all $i$. Assume that 
the codimension of $S_i$ in $V_i$ goes to infinity with $i$.
(Such vector bundles $V_i$ and closed subsets $S_i$ exist
because $X$ is a quotient stack. In more detail, if we write
$X$ as a quotient stack $Y/G$, then we can use
bundles $V_i$ which are given by suitable representations $V$ of $G$.
Take $V_i-S_i$ to be of the form $(Y\times (V-S))/G$ with
$(V-S)/G$ a quasi-projective scheme \cite[Remark 1.4]{Totarochow}.
Then
$(Y\times (V-S))/G$ is also a quasi-projective scheme
by \cite[Proposition 7.1]{GIT},
using that $Y$ has a $G$-equivariant ample line bundle.)

Define the motive $M^c(X)$ in $DM(k;R)$ to be the homotopy limit
of the sequence:
$$\cdots \arrow M^c(V_2-S_2)(-n_2)[-2n_2]
\arrow M^c(V_1-S_1)(-n_1)[-2n_1].$$
The morphisms here are the composition
\begin{align*}
M^c(V_{i+1}-S_{i+1})(-n_{i+1})[-2n_{i+1}] &\arrow
M^c(V_{i+1}-f_i^{-1}(S_i))(-n_{i+1})[-2n_{i+1}]\\
  &\cong M^c(V_i-S_i)(-n_i)[-2n_i],
\end{align*}
where the first morphism is the flat pullback associated
to an open inclusion, and the isomorphism follows from
homotopy invariance for affine bundles.
We will show that this motive is independent of the choice of
vector bundles
$V_i$ and closed substacks $S_i$.

Once we check that this motive is well defined
in Theorem \ref{motstack}, it will be immediate
that the motivic homology
of a quotient stack $X=Y/G$ given by the motive $M^c(X)$
agrees with the motivic homology of $X$ as defined by
Edidin and Graham \cite[sections 2.7 and 5.3]{EG}.
Namely, any given motivic homology
group $H_a(\cdot,R(b))$ of the sequence above is eventually constant.
In our notation, Edidin and Graham defined $H_a(X,R(b))$ to be equal to
$$H_a(((Y\times(V_j-S_j))/G)(-n_j)[-2n_j],R(b))\cong
H_{a+2n_j}((Y\times(V_j-S_j))/G,R(b+n_j))$$
for any $j$ sufficiently large.

Tudor Padurariu observed that for a smooth quotient stack $X$
of pure dimension $n$ over $k$, the motive $M(X)$ determines
$M^c(X)$ in a simple way: namely, $M^c(X)\cong M(X)^*(n)[2n]$.
In particular, $M^c(BG)\cong M(BG)^*(-\dim(G))[-2\, \dim(G)]$,
since $BG$ is a smooth stack of dimension $-\dim(G)$ over $k$.
For example, it follows that
$$CH_iBG\cong CH^{-\dim(G)-i}BG.$$
Padurariu's argument uses
that the dual of a direct sum in $DM(k;R)$ is a product,
and so the dual of a homotopy colimit is a homotopy limit. By contrast,
the dual of a product does not have a simple description
in general, and so it is not clear whether $M^c(X)^*$ 
is isomorphic to $M(X)(-n)[-2n]$ for a smooth quotient stack $X$
of pure dimension $n$ over $k$.

The following filtration of the category $DM(k;R)$ is
very convenient for our arguments.
Namely, for an integer $j$,
let $D_j(k;R)$ be the smallest localizing subcategory of $DM(k;R)$
that contains $M^c(X)(a)$ for all separated schemes $X$
of finite type over $k$
and all integers $a$ such that $\dim(X)+a\leq j$. (Another possible
notation would be $d_{\leq j}DM(k;R)$, by analogy with a notation
used for effective motives \cite[proof of Corollary 1.9]{HK}.)
Thus we have
a sequence of triangulated subcategories
$$\cdots \subset D_{-1} \subset D_0 \subset D_1 \subset \cdots$$
of $DM(k;R)$.

For an integer $j$, let $E_j$ be the smallest localizing
subcategory of $DM(k;R)$ that contains $M(Y)(a)$ for all smooth
projective varieties $Y$ over $k$ and all integers $a>j$. This is related
to the {\it slice filtration }of motives; in that setting,
$E_j$ would be called $DM^{\eff}(k;R)(j+1)$ \cite[section 1]{HK}.
For a triangulated subcategory $E$ of a triangulated category $T$,
the {\it right orthogonal }to $E$ is the full subcategory
$E^{\perp}$ of all objects $M$ such that $\Hom(N,M)=0$
for every $N$ in $E$ \cite{Orlov}.
The right orthogonal $E^{\perp}$ is always
a {\it colocalizing }subcategory of $T$, meaning a triangulated
subcategory that is closed under arbitrary products in $T$.
In the notation of the slice filtration, $E_j^{\perp}$ might be called
$\nu_{\leq j}DM(k;R)$ \cite[Definition 1.3]{HK}.

\begin{lemma}
\label{inclusion}
The subcategory $D_j$ of $DM(k;R)$ is contained in the right
orthogonal $E_j^{\perp}$.
\end{lemma}

\begin{proof}
As mentioned in section \ref{motives}, for any separated scheme
$Z$ of finite type over $k$, we have $H_j(Z,R(a))=0$
for all integers $a$ and $j$ with $a>\dim(Z)$. Let $Y$
be a smooth projective variety over $k$, and let $n=\dim(Y)$.
Then we have $H_j(Y\times Z, R(a))=0$ for all integers $a$
and $j$ with $a>n+\dim(Z)$. Equivalently,
$\Hom_{DM(k;R)}(R(a)[b],
\linebreak[0]
M(Y)\otimes M^c(Z))=0$ for all integers $a$ and $b$
with $a>n+\dim(Z)$.

As mentioned in section \ref{motives}, we have
$$M(Y)^*\cong M(Y)(-n)[-2n].$$
So $\Hom_{DM(k;R)}(R(a)[b],
\linebreak[0]
M(Y)^*\otimes M^c(Z))=0$
for all integers $a$ and $b$ with $a>\dim(Z)$. By Voevodsky
and Kelly's results (see the proof of Lemma \ref{inthom}),
it follows that $\Hom_{DM(k;R)}(M(Y)(a)[b],
\linebreak[0]
M^c(Z))=0$
for all integers $a$ and $b$ with $a>\dim(Z)$. Since the object
$M(Y)(a)$ is compact in $DM(k;R)$, it follows that
$\Hom_{DM(k;R)}(M(Y)(a)[b],
\linebreak[0]
N)=0$ for all motives $N$ in the subcategory
$D_j$ and all integers $a$ and $b$ such that $a>j$. 
Consequently, $D_j$ is contained in the right orthogonal $E_j^{\perp}$.
\end{proof}

Here is a convenient
formal property of the subcategories $E_j^{\perp}$.

\begin{lemma}
\label{perpcoloc}
For any integer $j$, the subcategory
$E_j^{\perp}$ of $DM(k;R)$ is both localizing and colocalizing.
That is, it is a triangulated subcategory of $DM(k;R)$
which is closed under arbitrary direct sums and arbitrary
products in $DM(k;R)$.
\end{lemma}

\begin{proof}
Since $E_j$ is a triangulated subcategory of $DM(k;R)$,
$E_j^{\perp}$ is a triangulated subcategory of $DM(k;R)$. As is true
for any
right orthogonal, $E_j^{\perp}$ is closed under arbitrary products
in $DM(k;R)$. Because $E_j$ is generated by a set of compact objects
in $DM(k;R)$, $E_j^{\perp}$ is also closed under arbitrary direct sums
in $DM(k;R)$ \cite[Theorem 5.1]{Neeman}.
\end{proof}

\begin{lemma}
\label{separating}
The intersection of the subcategories $E_j^{\perp}$ of $DM(k;R)$
for all integers $j$
is zero. It follows that the intersection of the subcategories $D_j$
for all integers $j$ is zero.
\end{lemma}

\begin{proof}
If a motive $N$ belongs to $E_j^{\perp}$ for all integers $j$,
then $\Hom_{DM(k;R)}(M(Y)(a)[b],
\linebreak[0]
N)=0$ for all smooth projective
varieties $Y$ over $k$ and all integers $a$ and $b$.
Since the triangulated
category $DM(k;R)$ is generated
by the objects $M^c(Y)(a)$ for smooth projective varieties $W$ over $k$
and integers $a$ (section \ref{motives}), it follows that $N=0$.
Thus $\cap_j E_j^{\perp}=0$. By Lemma \ref{inclusion},
it follows that $\cap_j D_j=0$.
\end{proof}

\begin{corollary}
\label{limitzero}
Let $\cdots \arrow X_2\arrow X_1$ be an inverse system
in $DM(k;R)$ such that $X_j$ is in the subcategory $D_{a_j}$
with $a_j\arrow -\infty$. Then $\holim_j X_j=0$.
\end{corollary}

\begin{proof}
We will show that the homotopy limit $X=\holim_j X_j$ belongs
to $E_m^{\perp}$ for every integer $m$, and hence is zero
by Lemma \ref{separating}. The homotopy limit does not
change if finitely many objects are removed from the inverse system.
So it suffices to show that if $X_j$ is in $D_{a_j}$ 
with $a_j\leq m$ for all $j$, then the homotopy limit $X$ is in $E_m^{\perp}$.
This is true because $D_m$ is contained in $E_m^{\perp}$ and
the triangulated subcategory $E_m^{\perp}$ is closed under arbitrary products
in $DM(k;R)$ (Lemma \ref{perpcoloc}).
\end{proof}

\begin{theorem}
\label{motstack}
The compactly supported motive in $DM(k;R)$
of a quotient stack over a field $k$ is an invariant
of the stack over $k$.
\end{theorem}

\begin{proof}
Let $X$ be a quotient stack over $k$.
Let $\cdots\surj V_2\surj V_1$ and $\cdots\surj W_2\surj W_1$
be two sequences of vector bundles over $X$, viewed as stacks
over $k$, with closed substacks $S_j\subset V_j$
and $T_j\subset W_j$ such that $V_j-S_j$ and $W_j-T_j$
are schemes, $S_{j+1}$ is contained in the inverse image
of $S_j$ under $V_{j+1}\surj V_j$ and likewise for $T_{j+1}$,
and the codimensions of $S_j\subset V_j$
and $T_j\subset W_j$ go to infinity. Let $m_j$ be the rank
of the bundle $V_j$ over $X$ and $n_j$ the rank of $W_j$.
We want to define
a canonical isomorphism from the motive 
$X_V:=\holim_j M^c(V_j-S_j)(-m_j)[-2m_j]$
to $X_W:=\holim_j M^c(W_j-T_j)(-n_j)[-2n_j]$.

Consider the sequence of vector bundles $V_j\oplus W_j$
over $X$, viewed as stacks over $k$. (These stacks are the fiber
products $V_j\times_X W_j$.) Let $Z_j$ be the union
of $S_j\times_X W_j$ and $V_j\times_X T_j$ inside $V_j\oplus W_j$.
Then we have flat morphisms of schemes from $(V_j\oplus W_j)-Z_j$
to $V_j-S_j$ and to $W_j-T_j$, for all $j$. So we have morphisms
from $X_V$ and from $X_W$ to
the homotopy limit $X_{VW}:=\holim_j M^c((V_j\oplus W_j)-Z_j)(-m_j-n_j)
[-2m_j-2n_j]$, as homotopy limits of flat pullback maps
of compactly supported motives.
It suffices to show that these morphisms
are both isomorphisms in $DM(k;R)$.

We will show that $X_V\arrow X_{VW}$ is an isomorphism;
the argument would be the same for $X_W$. The point is that
the morphism $(V_j\oplus W_j)-Z_j\arrow V_j-S_j$ is the complement
of the closed subset $V_j\times_X T_j$ in a vector bundle
(with fiber $W_j$) over the scheme $V_j-S_j$. The vector bundle
(of rank $n_j$) gives an isomorphism
$$M^c((V_j\oplus W_j)-(S_j\times_X W_j))(-m_j-n_j)[-2m_j-2n_j]
\cong M^c(V_j-S_j)(-m_j)[-2m_j].$$
Removing $T_j$ changes this motive by an object in the subcategory
$D_{-\codim(T_j\subset W_j)}$, by the localization triangle
for compactly supported motives (section \ref{motives}).
Therefore, the cone of the morphism
$X_V\arrow X_{VW}$ is a homotopy limit of motives in $D_{a_j}$
with $a_j$ approaching $-\infty$. By Corollary \ref{limitzero},
this cone is zero. That is, $X_V\arrow X_{VW}$ is an isomorphism,
as we want.
\end{proof}

Our definition of the compactly supported motive of a quotient
stack agrees with the standard definition in the special case
of a quasi-projective scheme.
As evidence that our definition is the right one for quotient stacks,
we prove a localization triangle.

\begin{theorem}
\label{localization}
Let $X$ be a quotient stack over a field $k$. Let $Y$ be a closed
substack of $X$. Then there is a distinguished triangle
$$M^c(Y)\arrow M^c(X)\arrow M^c(X-Y)$$
in $DM(k;R)$.
\end{theorem}

\begin{proof}
Write $X$ as the quotient stack $A/G$
for a quasi-projective scheme $A$ over $k$ and
an affine group scheme $G$ of finite type over $k$ with
a $G$-equivariant ample line bundle on $A$.
Let $\cdots\surj V_2 \surj V_1$ be a sequence of representations
of $G$. Let $n_i=\dim(V_i)$. We can choose these representations
so that there are closed subschemes $S_i\subset V_i$
such that $G$ acts freely on $V_i-S_i$ with quotient
a quasi-projective scheme over $k$, $S_{i+1}$ is contained
in the inverse image of $S_i$ for all $i$, and
the codimension of $S_i$ in $V_i$ goes to infinity with $i$.
Let $A_Y$ and $A_{X-Y}$ be the inverse images of $Y$ and $X-Y$
in the scheme $A$.
Then the distinguished triangle we want is the homotopy limit
of the distinguished triangles
$$M^c((A_Y\times (V_i-S_i))/G)\arrow M^c((A\times (V_i-S_i))/G)
\arrow M^c(A_{X-Y}\times (V_i-S_i))/G).$$
\end{proof}

We now describe a basic example of the motive of a quotient stack,
$M^c(BG_m)$.

\begin{lemma}
\label{bgm}
The compactly supported motive of $BG_m$ in $DM(k;R)$ is isomorphic
to $\prod_{j\leq -1} R(j)[2j]$. This is isomorphic to the direct
sum $\oplus_{j\leq -1} R(j)[2j]$.
\end{lemma}

\begin{proof}
By definition, using the representation of $G_m$ by scalars
on an $n$-dimensional vector space for any given $n$,
$M^c(BG_m)$ is the homotopy limit of the motives 
$$M^c(\P^{n-1})(-n)[-2n]\cong \prod_{j=-n}^{-1} R(j)[2j],$$
and so $M^c(BG_m)$
is isomorphic to the product $\prod_{j\leq -1} R(j)[2j]$.

To show that the morphism from the direct sum
$\oplus_{j\leq -1} R(j)[2j]$ to the product is an isomorphism,
it suffices to show that the cone $N$ of this morphism is zero.
For any integer $a<0$, $N$ is isomorphic to the cone
of the morphism $\oplus_{j\leq a} R(j)[2j]\arrow
\prod_{j\leq a} R(j)[2j]$, because finite direct sums are the same
as finite products. Because the category $E_a^{\perp}$ is both
localizing and colocalizing (Lemma \ref{perpcoloc}), both
$\oplus_{j\leq a} R(j)[2j]$ and $\prod_{j\leq -1} R(j)[2j]$
are in $E_a^{\perp}$.
So $N$ is in $E_a^{\perp}$. Since this holds for all negative
integers $a$, $N$ is zero by Lemma \ref{separating}, as we want.
\end{proof}

\begin{lemma}
\label{dimbound}
Let $X$ be a quotient stack over a field $k$.
Then the motive $M^c(X)$ is in the subcategory $(E_{\dim(X)})^{\perp}$. 
That is, for
every smooth projective variety $Y$ over $k$,
$$\Hom_{DM(k;R)}(M(Y)(a)[b],M^c(X))=0$$
for all integers $a$ and $b$ with $a>\dim(X)$.
\end{lemma}

For a quotient stack $X$,
one might ask whether $M^c(X)$ is always
in the subcategory
$D_{\dim(X)}$. For example, that is true for $M^c(BG_m)=\prod_{j\leq -1}
\Z(j)[2j]$, because that is isomorphic to $\oplus_{j\leq -1}\Z(j)[2j]$
by Lemma \ref{bgm},
and that direct sum
is in $D_{-1}$. It would be clear that the compactly supported
motive of a quotient stack $X$ was in $D_{\dim(X)}$ if the categories
$D_m$ were closed under arbitrary products in $DM(k;R)$, but I suspect
that they are not.

\begin{proof}
As mentioned in section \ref{motives}, for any separated scheme
$Z$ of finite type over $k$, we have $H_j(Z,R(a))=0$
for all integers $a$ and $j$ with $a>\dim(Z)$. Let $Y$
be a smooth projective variety over $k$, and let $n=\dim(Y)$.
Then we have $H_j(Y\times Z, R(a))=0$ for all integers $a$
and $j$ with $a>n+\dim(Z)$. Equivalently,
$\Hom_{DM(k;R)}(R(a)[b],
\linebreak[0]
M(Y)\otimes M^c(Z))=0$ for all integers $a$ and $b$
with $a>n+\dim(Z)$.

As mentioned in section \ref{motives}, we have
$$M(Y)^*\cong M(Y)(-n)[-2n].$$
So $\Hom_{DM(k;R)}(R(a)[b],
\linebreak[0]
M(Y)^*\otimes M^c(Z))=0$
for all integers $a$ and $b$ with $a>\dim(Z)$. By the results of Voevodsky
and Kelly (see the proof of Lemma \ref{inthom}),
it follows that $\Hom_{DM(k;R)}(M(Y)(a)[b],
\linebreak[0]
M^c(Z))=0$
for all integers $a$ and $b$ with $a>\dim(Z)$.

Let $X$ be a quotient stack over $k$.
By definition, $M^c(X)$ is a homotopy limit
of motives $M^c(V-S)(-m)[-2m]$, where $V$ is a vector 
bundle over $X$ (viewed as a stack over $k$) and $V-S$
is an open subscheme. The scheme
$V-S$ has dimension $\dim(X)+m$.
So the previous paragraph gives that 
$$\Hom_{DM(k;R)}
(M(Y)(a)[b],M^c(V-S)(-m)[-2m])=0$$
for all integers $a$ and $b$ with $a>\dim(X)$. By definition of
$M^c(X)$ as a 
homotopy limit, it follows that
$$\Hom_{DM(k;R)}(M(Y)(a)[b],M^c(X))=0$$
for all integers $a$ and $b$ with $a>\dim(X)$. That is,
$M^c(X)$ is in the subcategory $(E_{\dim(X)})^{\perp}$.
\end{proof}

\begin{lemma}
\label{coloc}
Let $X$ be a motive in the subcategory $E_m^{\perp}$ of $DM(k;R)$
for an integer $m$.
Then the colocalization $C(X)$ with respect to Tate motives
is in the subcategory $D_m$, and hence in $E_m^{\perp}$.
\end{lemma}

\begin{proof}
We use the construction of $C(X)$ 
from section \ref{kunnethsect} as a homotopy colimit
$\hocolim_j\, C_j$. Since $X$ is in $E_m^{\perp}$, we have
$H_b(X,R(a))=0$ for all integers $a$ and $b$ with $a>m$.
So we can take the motive $C_0$ in the construction of $C(X)$
to be a direct sum of motives $R(a)[b]$ with $a\leq m$.
Then $C_0$ is in $D_m$. So $H_b(C_0,R(a))=0$
for all integers $a$ and $b$ with $a>m$. By induction, we can choose
$C_j$ for all natural numbers $j$ to be in $D_m$. So
$C(X)=\hocolim_j\, C_j$ is in $D_m$. By Lemma \ref{inclusion},
$C(X)$ is also in $E_m^{\perp}$.
\end{proof}

Define a motive $A$ in $DM(k;R)$ to be {\it mixed Tate
modulo dimension $m$ }if the cone of the morphism $C(A)\arrow A$
is in $E_m^{\perp}$. Also, define a quotient stack $X$ to be {\it mixed Tate
modulo codimension $r$ }if $M^c(X)$ is mixed Tate modulo dimension
$\dim(X)-r$.

\begin{lemma}
\label{Tatemodulo}
All mixed Tate motives and all motives in $E_m^{\perp}$ are
mixed Tate
modulo dimension $m$. Also, the motives that are mixed Tate modulo
dimension $m$ form a triangulated subcategory of $DM(k;R)$.
\end{lemma}

\begin{proof}
It is clear that a mixed Tate motive is mixed Tate modulo dimension $m$.
Also, a motive in $E_m^{\perp}$ is mixed Tate modulo dimension $m$,
by Lemma \ref{coloc}. It remains to show that for a distinguished
triangle $X\arrow Y\arrow Z$ in $DM(k;R)$ with $X$ and $Y$ mixed Tate
modulo dimension $m$, $Z$ is also mixed Tate modulo dimension $m$.
We have a morphism of distinguished triangles:
$$\xymatrix@C+5pt@R-10pt{
C(X)\ar[r]\ar[d] & C(Y)\ar[r]\ar[d] & C(Z)\ar[d]\\
X\ar[r] & Y\ar[r] & Z.
}$$
By the octahedral axiom for triangulated categories,
the cone of $C(Z)\arrow Z$ is the cone of a morphism
$\cone(C(X)\arrow X)\arrow \cone(C(Y)\arrow Y)$. The latter two
cones are in $E_m^{\perp}$, and so the cone of $C(Z)\arrow Z$
is also in $E_m^{\perp}$. That is, $Z$ is mixed Tate modulo
dimension $m$.
\end{proof}

\begin{corollary}
\label{approx}
Let $X$ be a motive in $DM(k;R)$ which can be approximated
by mixed Tate motives in the sense that $X$ is mixed Tate modulo
dimension $j$ for every integer $j$.
Then $X$ is a mixed Tate motive.
\end{corollary}

\begin{proof}
The cone of $C(X)\arrow X$ is in $E_j^{\perp}$ for every integer $j$,
and hence is zero by Lemma \ref{separating}.
\end{proof}

Given more geometric information on a motive $N$, the following results
give better criteria for when $N$ is mixed Tate.

\begin{lemma}
\label{minusone}
Let $X$ be a separated scheme of finite type over $k$.
If $X$ is mixed Tate modulo dimension $-1$, then $X$ is mixed Tate.
\end{lemma}

\begin{proof}
The motive $M^c(X)$ is in the subcategory $E_{-1}$
of effective motives in $DM(k;R)$, under our assumption
on $R$ \cite[Proposition 5.5.5]{Kelly}. Let $W$ be the cone
of the morphism $C(M^c(X))\arrow M^c(X)$.
Our assumption that $X$ is mixed Tate modulo
dimension $-1$ means that $W$ is in $E_{-1}^{\perp}$.
So the morphism $M^c(X)\arrow W$ is zero. It follows
that $M^c(X)$ is a summand of the mixed Tate motive $C(M^c(X))$.
So $M^c(X)$ is a mixed Tate motive.
\end{proof}

There is a ``finite-dimensional'' criterion for when a quotient stack
is mixed Tate, Corollary \ref{finitedim}.
Namely, a quotient stack $X=Y/G$ over $k$ is mixed Tate
(meaning that $M^c(X)$ is mixed Tate in $DM(k;R)$) if and only if
the scheme $(Y\times GL(n))/G$ is mixed Tate, for one or any
faithful representation $G\inj GL(n)$ over $k$.

Here is the main step in proving that.

\begin{lemma}
\label{glnmodulo}
Let $X$ be a quotient stack over a field $k$.
Let $E$ be a principal $GL(n)$-bundle over $X$ for some $n$, viewed
as a stack over $k$. Let $r$ be an integer.
Then $X$
is mixed Tate modulo codimension $r$ (in $DM(k;R)$) if and only if
$E$ is mixed Tate modulo codimension $r$.
\end{lemma}

\begin{proof}
First consider the case $n=1$, so that $E$ is a principal
$G_m$-bundle over $X$. Think of $E$ as the complement of the zero
section in a line bundle over $X$. The localization triangle has the form
$$M^c(X)\arrow M^c(X)(1)[2]\arrow M^c(E)$$
in $DM(k;R)$. Consider the morphism of distinguished triangles:
$$\xymatrix@C+5pt@R-10pt{
C(M^c(X))\ar[r]\ar[d] & C(M^c(X))(1)[2]\ar[r]\ar[d] & C(M^c(E))\ar[d]\\
M^c(X)\ar[r] & M^c(X)(1)[2]\ar[r] & M^c(E)
}$$
Let $W$ be the cone of $C(M^c(X))\arrow M^c(X)$
and let $N$ be the cone
of $C(M^c(E))\arrow M^c(E)$. The diagram gives
a distinguished triangle $W\arrow W(1)[2]\arrow N$.

If $X$ is mixed Tate modulo codimension $r$,
then $W$ is in $E_{\dim(X)-r}^{\perp}$.
So $W(1)[2]$ is in $E_{\dim(X)+1-r}^{\perp}$, and hence
$N$ is in $E_{\dim(X)+1-r}^{\perp}=E_{\dim(E)-r}^{\perp}$.
That is, the stack $E$ is mixed Tate modulo codimension $r$,
as we want.

Conversely, suppose that $E$ is mixed Tate modulo codimension $r$.
That is, $N$ is in $E_{\dim(E)-r}^{\perp}=E_{\dim(X)+1-r}^{\perp}$.
By Lemma \ref{dimbound}, $X$ is in $E_{\dim(X)}^{\perp}$.
By Lemma \ref{coloc}, $C(X)$ is also in $E_{\dim(X)}^{\perp}$, and hence
$W$ is in $E_{\dim(X)}^{\perp}$. We want to show that $X$
is mixed Tate modulo codimension $r$, meaning that $W$
is in $E_{\dim(X)-r}^{\perp}$. If not, then there is a smallest
integer $j$ such that $W$ is in $E_j^{\perp}$; we have $j>\dim(X)-r$
by assumption. Then $W(-1)$ is in $E_{j-1}^{\perp}$.
The distinguished triangle 
$$W(-1)[-2]\arrow W\arrow N(-1)[-2]$$
gives that $W$ is in $E_{j-1}^{\perp}$, a contradiction.
Thus $X$ is mixed Tate modulo codimension $r$ if and only if
the principal $G_m$-bundle $E$ over $X$
is mixed Tate modulo codimension $r$.

Now let $E$ be a principal $GL(n)$-bundle over a stack $X$,
with $n$ arbitrary. Let $B$ be the subgroup of upper-triangular
matrices in $GL(n)$ over $k$. Then $E/B$ is an iterated projective
bundle over $X$, and so
$$M^c(E/B)\cong \oplus_j M^c(X)(a_j)[2a_j],$$
where $a_1,\ldots,a_{n!}$ are the dimensions of the Bruhat cells
of the flag manifold $GL(n)/B$. Assume that
$X$ is mixed Tate modulo codimension $r$, that is, modulo
dimension $\dim(X)-r$. Then $M^c(X)(a)[2a]$ is mixed Tate modulo
dimension $\dim(X)-r+a$, for any integer $a$. It follows that
$E/B$ is mixed Tate modulo dimension $\dim(X)-r+\dim(G/B)=
\dim(E/B)-r$. That is, $E/B$ is mixed Tate modulo codimension $r$.
Conversely, if $E/B$ is mixed Tate modulo codimension $r$,
then the summand $M^c(X)(\dim(G/B))[2\, \dim(G/B)]$ of $M^c(E/B)$
is mixed Tate modulo dimension $\dim(E/B)-r=\dim(X)+\dim(G/B)-r$,
and so $M^c(X)$ is mixed Tate modulo dimension $\dim(X)-r$,
thus modulo codimension $r$.

Next, let $U$ be the subgroup of strictly upper-triangular
matrices in $GL(n)$ over $k$. Since $B/U\cong (G_m)^n$,
the stack $E/U$ is a principal $(G_m)^n$-bundle over $E/B$.
Applying our result on principal $G_m$-bundles $n$ times,
we deduce that $E/U$ is mixed Tate modulo codimension $r$
if and only if $E/B$ is mixed Tate modulo codimension $r$,
hence if and only if $X$ is mixed Tate modulo codimension $r$. Finally,
$U$ is an extension of copies of the additive group,
and so homotopy invariance gives that
$$M^c(E)\cong M^c(E/U)(\dim(U))[2\, \dim(U)].$$
It follows that $E$ is mixed Tate modulo codimension $r$ if and only if
$X$ is mixed Tate modulo codimension $r$.
\end{proof}

\begin{corollary}
\label{gln}
Let $X$ be a quotient stack over a field $k$.
Let $E$ be a principal $GL(n)$-bundle over $X$ for some $n$, viewed
as a stack over $k$. Then $E$
is mixed Tate (in $DM(k;R)$)
if and only if $X$ is mixed Tate.
\end{corollary}

\begin{proof}
This follows from Lemma \ref{glnmodulo}, since a motive is mixed
Tate if and only if it is mixed Tate modulo dimension $r$ for all
integers $r$ (Corollary \ref{approx}).
\end{proof}

\begin{corollary}
\label{finitedim}
Let $Y$ be a quasi-projective scheme over a field $k$ and $G$
an affine group scheme of finite type over $k$ that acts on $Y$
such that there is a $G$-equivariant ample line bundle on $Y$.
Let $G\inj GL(n)$ be a faithful representation of $G$ over $k$.
Then (the compactly supported motive
of) the stack $Y/G$ over $k$ is mixed Tate if and only if the scheme
$(Y\times GL(n))/G$ over $k$ is mixed Tate.
\end{corollary}

\begin{proof}
The scheme $(Y\times GL(n))/G$ is a principal $GL(n)$-bundle
over the stack $Y/G$. So this follows from Corollary \ref{gln}.
\end{proof}

For example, $BG$ is mixed Tate if and only if the scheme
$GL(n)/G$ is mixed Tate, for one or any faithful representation
$G\inj GL(n)$ over $k$.

As a result, we now show that
the structure of a classifying space $BG$ is determined in some ways
by its properties in low codimension, namely codimension $n^2$ (roughly),
where $n$ is the dimension of a faithful representation of $G$.
Theorem \ref{reduceBG} reduces the question of whether $BG$
is mixed Tate even further,
to properties in codimension $n$ (roughly) together with properties
of subgroups of $G$.

\begin{theorem}
\label{finitedimapprox}
Let $G$ be an affine group scheme over a field $k$. Suppose that
$G$ has a faithful representation of dimension $n$ over $k$.
If $BG$ is mixed Tate in $DM(k;R)$ modulo codimension $n^2-\dim(G)+1$,
then $BG$ is mixed Tate in $DM(k;R)$.
\end{theorem}

\begin{proof}
We have a principal $GL(n)$-bundle $GL(n)/G\arrow BG$ of stacks
over $k$. By Lemma \ref{glnmodulo}, if $BG$ is mixed Tate modulo
codimension $n^2-\dim(G)+1$, then the variety
$GL(n)/G$ is also mixed Tate modulo
codimension $n^2-\dim(G)+1$. 
Since $GL(n)/G$ has dimension $n^2-\dim(G)$,
Lemma \ref{minusone} gives that $GL(n)/G$ is mixed Tate. By
Corollary \ref{finitedim}, $BG$ is mixed Tate.
\end{proof}

\section{The mixed Tate property for classifying spaces}
\label{TateBGsect}

The work of Bogomolov and Saltman defines a dichotomy
among all finite groups $G$: is $BG_{\C}$ stably rational?
(This means that the variety
$V/G$ is stably rational for one, or any, faithful
representation $V$ of $G$ over $\C$.) This paper has considered
several other dichotomies among finite groups $G$. Is the birational
motive of $BG_{\C}$ trivial? Does $BG_{\C}$ have the weak
or strong Chow K\"unneth property? It would be interesting
to know whether these conditions are all equivalent.

Ekedahl defined another property with the same flavor,
for a finite group scheme $G$ over a field $k$.
Namely, when does the stack $BG$ have the class of a point
in the ring $A=K_0(\Var_k)[L^{-1},(L^n-1)^{-1}:\, n\geq 1]$?
Here $K_0(\Var_k)$ denotes the Grothendieck ring of $k$-varieties
and $L$ is the class of $A^1$.
Ekedahl showed that
this property is equivalent to the statement (not mentioning
stacks) that for one or any faithful representation $G\inj GL(n)$,
the variety $GL(n)/G$ is equal to $GL(n)$ in the ring $A$
\cite[Proposition 3.1]{Ekedahl}. I do not know any implications
between Ekedahl's property and the other properties we have mentioned,
but it may be that all these
properties are equivalent. In particular, Ekedahl's 
property fails if $G$ has nontrivial unramified $H^2$
\cite[Theorem 5.1]{Ekedahl}; for such groups, all the properties
we have mentioned fail.

In this section, we consider another dichotomy among finite
groups, or more generally among affine group schemes $G$:
is $BG$ mixed Tate, meaning that the motive $M^c(BG)$ is mixed
Tate? This property is equivalent to the motivic K\"unneth
property formulated in the introduction to section \ref{quotientsect}.
It implies the
Chow K\"unneth property, since it gives information
about all of motivic homology, not just Chow groups.
The mixed Tate property may be equivalent
to all the other properties mentioned above.

We have examples of finite groups which are not mixed Tate
(say over $\C$),
because they do not even have
the weak Chow K\"unneth property (Corollary \ref{BG}).
To justify the concept, we will also give examples
of finite groups which are mixed Tate: the symmetric
groups (Theorem \ref{symmetric}), the finite general linear
groups in cross-characteristic (Theorem \ref{gl}),
and all finite subgroups
of $GL(2)$ (Corollary \ref{dim2}).
It is conceivable that all ``naturally occurring'' finite groups
are mixed Tate over $\C$. For example, Bogomolov
conjectured that for every finite simple group $G$, quotient
varieties $V/G$ are stably rational \cite{Bogomolovstable}. In that
direction, Kunyavskii showed that every finite simple group
has unramified $H^2$ equal to zero \cite{Kunyavskii}.
Likewise, I conjecture that all finite simple groups
are mixed Tate. More generally, all quasisimple or almost simple
groups should be mixed Tate.

In order to give examples of finite groups which are mixed Tate,
we start by proving some formal properties of mixed Tate stacks.
By Corollary \ref{finitedim},
$BG$ is mixed Tate if and only if the variety
$GL(n)/G$ is mixed Tate
for a faithful representation $V$ of $G$ with $\dim(V)=n$. But
$GL(n)/G$ may be hard to analyze because it has high dimension, namely
$n^2$. Theorem \ref{reduceBG} gives
a sufficient condition for $BG$ to be mixed
Tate in terms of the variety $(V-S)/G$, which has dimension only $n$,
together with information on subgroups of $G$.

Throughout this section, we work in the category $DM(k;R)$ for a field $k$
and a commutative ring $R$ in which the exponential characteristic
of $k$ is invertible.

\begin{lemma}
Let $X$ be a quotient stack over a field $k$ and $Y$ a closed substack.
If two of $X$, $Y$, $X-Y$ are mixed Tate, then so is the third.
\end{lemma}

\begin{proof}
This follows from the localization triangle
$$M^c(Y)\arrow M^c(X)\arrow M^c(X-Y)$$
(Theorem \ref{localization}).
\end{proof}

\begin{lemma}
\label{localcoefficients}
Let $k$ be a field, and let $e$ be the exponential characteristic of $k$.
A quotient stack $X$ over a field $k$ is mixed Tate with $\Z[1/e]$ coefficients
(that is, in $DM(k;\Z[1/e])$)
if and only if it is mixed Tate with $\Z_{(p)}$ coefficients
for all prime numbers $p$ that are invertible in $k$.
\end{lemma}

\begin{proof}
Write $X$ as the quotient stack $A/G$ for some affine group scheme $G$
of finite type over $k$ and some quasi-projective scheme $A$ over $k$
with a $G$-equivariant ample line bundle. Let $G\inj GL(n)$
be a faithful representation over $k$. Then $E=(A\times GL(n))/G$
is a quasi-projective scheme over $k$, and $GL(n)$ acts on $E$
with quotient stack $E/GL(n)\cong X$.
By Corollary \ref{gln}, $M^c(E)$ is mixed Tate
(with any coefficients) if and only if $M^c(X)$ is mixed Tate.
So it suffices to show that $M^c(E)$ is mixed Tate
in $DM(k;\Z[1/e])$ if and only if it is mixed Tate in $DM(k;\Z_{(p)})$
for all prime numbers $p$ that are invertible in $k$.

For a commutative ring $R$, $E$ is $R$-mixed Tate if and only if it
has the K\"unneth property for the $R$-motivic homology of $E\times Y$
for all separated $k$-schemes $Y$ of finite type (Theorem \ref{main}).
The motivic homology with $R$ coefficients of a $k$-scheme
is related to motivic homology with $\Z$ coefficients by the universal
coefficient theorem. Let $p$ be a prime number that is invertible
in $k$. Since $\Z_{(p)}$ and $\Z[1/e]$ are flat over $\Z$,
the K\"unneth spectral sequence for $E\times Y$ with 
$\Z_{(p)}$ coefficients is just the localization at $p$
of the spectral sequence with $\Z[1/e]$ coefficients. A homomorphism
of $\Z[1/e]$-modules is an isomorphism if and only if it is an
isomorphism $p$-locally for all prime numbers $p$ that are invertible
in $k$. Therefore,
$X$ is $\Z[1/e]$-mixed Tate if and only if it is $\Z_{(p)}$-mixed
Tate for all prime numbers $p$ that are invertible in $k$.
\end{proof}

\begin{lemma}
\label{sylowtate}
Let $G$ be a finite group, $p$ a prime number, and $H$ a Sylow
$p$-subgroup of $G$. Fix a base field $k$ in which $p$ is invertible.
Let $R$ be the ring
$\Z/p$ or $\Z_{(p)}$. If $BH$ is $R$-mixed
Tate, then $BG$ is $R$-mixed Tate.
\end{lemma}

\begin{proof}
Use that $BG$ is $R$-mixed Tate if and only if it has 
the K\"unneth property for $BG\times Y$ for all $k$-schemes
$Y$ of finite type. Let $R$ be $\Z/p$ or $\Z_{(p)}$.
Using the transfer, the K\"unneth
spectral sequence for $BG\times Y$ is a summand with $R$ coefficients
of the spectral sequence for $BH\times Y$. Therefore, if $BH$
satisfies the motivic K\"unneth property with $R$ coefficients,
then so does $BG$.
\end{proof}

For a representation $V$ of a finite group $G$ and $K$ a subgroup
of $G$, $V^K$ means the linear subspace fixed by $K$. Following
Ekedahl \cite{Ekedahl}, let $V_K$ be the open subset of $V^K$
of points with stabilizer in $G$ equal to $K$, meaning that
$V_K=V^K-\cup_{K\subsetneqq L}V^L$.

\begin{lemma}
\label{glcut}
Let $s$ be a natural number.
Let $V$ be a faithful representation of a finite group $G$ over
a field $k$.
For each subgroup
$K$ of $G$ that occurs as the stabilizer of a point
in $V$, assume that the stack $V_K/N_G(K)$ is mixed Tate in $DM(k;R)$
modulo codimension $s-\codim(V^K\subset V)$.
Then $BG$ is mixed Tate modulo codimension $s$.
\end{lemma}

\begin{proof}
The stack $V/G$ is a vector bundle over $BG$. So if we can show that
the stack $V/G$ is mixed Tate modulo codimension $s$, then $BG$
is mixed Tate modulo codimension $s$, as we want.

The stack $V/G$ is the disjoint union
of the locally closed substacks $V_K/N_G(K)$ for all conjugacy
classes of stabilizer subgroups
$K$ of $G$. By assumption, each substack $V_K/N_G(K)$ is mixed Tate
modulo codimension $s-\codim(V^K\subset V)$, that is,
modulo dimension $\dim(V)-\dim(G)-s$. The motives that
are mixed Tate modulo dimension $\dim(V)-s$ form a triangulated
category (Lemma \ref{Tatemodulo}).
By the localization triangle
for stacks (Theorem \ref{localization}), the stack $V/G$
is mixed Tate modulo dimension $\dim(V)-s$, that is, modulo
codimension $s$.
\end{proof}

A next step is to express the assumptions on smaller groups in terms
of classifying spaces, as follows. This step
may not be needed in some examples,
but it leads to a neat statement, Theorem \ref{reduceBG}.
(We will apply Lemma \ref{backtoBG}
to the subgroups
$H=N_G(K)$ acting on $V^K$ in Lemma \ref{glcut}, typically not faithfully.)

\begin{lemma}
\label{backtoBG}
Let $s$ be a natural number.
Let $V$ be a representation of a finite group $H$ over a field $k$,
not necessarily faithful. Let $K_1=\ker(H\arrow GL(V))$.
Consider all chains $K_1\subsetneqq K_2\subsetneqq \cdots
\subsetneqq K_r \subset H$, $r\geq 1$, such that if we define
$N_i=\cap_{j\leq i}N_H(K_j)\subset H$, then $K_{i+1}$ is the stabilizer
of a point for $N_i$ acting on $V^{K_i}$. For every such chain,
assume that $BN_r$ is mixed Tate in $DM(k;R)$ modulo codimension $s$.
(In particular, for $r=1$,
we are assuming that $BH$ is mixed Tate modulo codimension $s$.)
Then the stack $V_{K_1}/H$ is mixed Tate modulo codimension $s$.
\end{lemma}

\begin{proof}
By our assumption (with $r=1$), the stack $BH$ is mixed Tate
modulo codimension $s$.
So the stack $V/H$ (a vector bundle over $BH$) is mixed Tate
modulo codimension $s$.
The difference $V/H-V_{K_1}/H$ is the disjoint union of the locally
closed substacks $(\coprod_{g\in H/N_H(K_2)} V_{gK_2 g^{-1}})/H$
for conjugacy classes of stabilizer subgroups $K_2$ for $H$ acting 
on $V$ with $K_1\subsetneqq K_2$. That quotient
is isomorphic to the stack $V_{K_2}/N_H(K_2)$. By our assumption (with
$r=2$), $BN_H(K_2)=BN_2$ is mixed Tate modulo codimension $s$,
and so the stack
$V^{K_2}/N_H(K_2)$ (a vector bundle over $BN_2$) is also mixed Tate
modulo codimension $s$.
The stack we want is the open substack $V_{K_2}/N_H(K_2)$
of $V^{K_2}/N_H(K_2)$. The complement is the disjoint union
of the locally closed substacks 
$$\bigg( \coprod_{g\in N_2/N_3} V_{gK_3 g^{-1}}\bigg) 
/N_2 \cong V_{K_3}/N_3,$$
where $K_3$ runs over all stabilizer subgroups for $N_2$ acting
on $V^{K_2}$ with $K_1\subsetneqq K_2\subsetneqq K_3$,
and $N_{N_2}(K_3)=\cap_{j\leq 3}N_H(K_j)=N_3$. Since $H$ is finite,
the process stops
after finitely many steps and gives the statement of the lemma.
\end{proof}

Combining the previous two lemmas gives the following result.
Theorem \ref{reduceBG} shows that $BG$ is mixed
Tate if the variety $V_1/G$ is mixed Tate and $BH$ is mixed Tate for
certain proper subgroups $H$ of $G$. (As in the notation above,
$V_1$ denotes the open subset of $V$ where $G$ acts freely.)
Theorem \ref{reduceBG}
was suggested by a similar statement by Ekedahl about his
invariant of $BG$ in the Grothendieck ring of varieties
\cite[Theorem 3.4]{Ekedahl}, but I do not see a direct implication
between the two results.

\begin{theorem}
\label{reduceBG}
Let $V$ be a faithful representation of a finite group $G$
over a field $k$. Consider
all chains $1=K_0\subsetneqq K_1\subsetneqq \cdots \subsetneqq
K_r\subset G$, $r\geq 1$, such that if we define
$N_i=\cap_{j\leq i}N_G(K_j)\subset G$, $K_{i+1}$ is a stabilizer
subgroup for $N_i$ acting on $V^{K_i}$. Suppose
that the variety $V_1/G$ is mixed Tate in $DM(k;R)$
and that the stack $BN_r$
is mixed Tate for all such chains with $N_r\neq G$.
Then $BG$ is mixed Tate.
\end{theorem}

\begin{proof}
We show by induction on $s$ that $BG$ is mixed
Tate modulo codimension $s$ for every natural number $s$. That will
imply that $BG$ is mixed Tate by Corollary \ref{approx} (or by
the stronger Theorem \ref{finitedimapprox}).
Clearly $BG$ is mixed Tate modulo codimension 0.
Suppose that $BG$ is mixed Tate modulo codimension $s$. To show
that $BG$ is mixed Tate modulo codimension $s+1$, we use Lemma
\ref{glcut}. So it suffices to show that for each stabilizer subgroup
$K_1$ of $G$ acting on $V$, the stack $V_{K_1}/N_G(K_1)$
is mixed Tate modulo codimension $s+1-\codim(V^{K_1}\subset V)$.
For $K_1=1$, this is true, because we assume that the variety $V_1/G$ 
is mixed Tate. It remains to consider a stabilizer subgroup $K_1\neq 1$.
We apply Lemma \ref{backtoBG} to the vector space $V^{K_1}$
with its action of $N_G(K_1)$. If $N_G(K_1)\neq G$, then Lemma
\ref{backtoBG} and our assumptions imply that the stack
$V_{K_1}/N_G(K_1)$ is mixed Tate. Finally, if $K_1\neq 1$ and
$N_G(K_1)=G$, then Lemma \ref{backtoBG}, our assumptions, and the inductive
hypothesis that $BG$ is mixed Tate modulo codimension $s$ imply
that the stack $V_{K_1}/N_G(K_1)$ is mixed Tate modulo codimension $s$.
This implies that $V_{K_1}/N_G(K_1)$ is mixed Tate modulo codimension
$s+1-\codim(V^{K_1}\subset V)$ (as we want), because $\codim(V^{K_1}\subset
V)>0$, since $K_1\neq 1$ and $G$ acts faithfully on $V$. The induction
is complete. So $BG$ is mixed Tate.
\end{proof}

We now use Theorem \ref{reduceBG} to give examples of finite groups
which are mixed Tate. (The assumption on the field $k$ in Corollary
\ref{dim2} could be weakened.) For example, Corollary
\ref{dim2} gives that the dihedral groups,
generalized quaternion groups, modular 2-groups, and semidihedral groups
\cite[section 23.4]{Aschbacher} are mixed Tate.

\begin{corollary}
\label{dim2}
Let $k$ be a field that contains $\overline{\Q}$. Let $G$ 
be a finite subgroup of $GL(2)$ over $k$. Then $BG$ is mixed Tate
in $DM(k;\Z)$.
\end{corollary}

\begin{proof}
Use induction on the order of $G$. Let $V$ be the given 2-dimensional
faithful representation of $G$. Since $BH$ is mixed Tate
for all proper subgroups $H$ of $G$, Theorem \ref{reduceBG} shows
that $BG$ is mixed Tate if the variety $V_1/G$ is mixed Tate.

The group $G$ acts on the projective space $\P^1$ of lines in $V_1$.
The coarse quotient $\P^1/G$ is a normal projective curve
over $k$, and so it is smooth over $k$. It is unirational over $k$,
and hence isomorphic to $\P^1$ over $k$.

It is convenient to observe that the representation $V$ of $G$
can be defined over $\overline{\Q}$. Let $S$ be the closed subset
of $\P^1$ where $G$ does not act freely; then $(\P^1-S)/G$ is isomorphic
to $\P^1-T$ for some closed subset $T$. Since $S$ and $T$ are defined
over $\overline{\Q}$, $T$ is a finite union of copies of $\Spec(k)$.
So $\P^1-T$ is a linear scheme over $k$ (as defined in section
\ref{motives}). An open subset
of $V_1/G$ is a principal $G_m$-bundle over $\P^1-T$, and hence
is a linear scheme over $k$. The complement of this open subset
is the union of finitely many curves of the form $G_m/H$ where $H$
is a finite subgroup of $G_m$; these are isomorphic to $G_m$
and hence are linear schemes over $k$. So $V_1/G$ is a linear
scheme over $k$. Thus $V_1/G$ is mixed Tate, and so $BG$ is mixed Tate.
\end{proof}

We now show that many wreath product groups are mixed Tate.
It will follow that the finite general linear
groups in cross-characteristic and the symmetric groups
are mixed Tate (Theorems \ref{symmetric} and \ref{gl}),
since their Sylow $p$-subgroups
are products of iterated wreath products of cyclic groups.
This is related to Voevodsky's construction of Steenrod operations
on motivic cohomology, which can be viewed as computing
the motivic cohomology of the symmetric 
groups over any field \cite[section 6]{Voevodskyop}, \cite{Voevodskyoper}.

\begin{lemma}
\label{cyclicproduct}
Let $k$ be a field of characteristic not $p$ that contains
the $p$th roots of unity.
Let $X$ be a quasi-projective linear scheme over $k$ (as defined
in section \ref{motives}). Then the cyclic product
$Z^pX=X^p/(\Z/p)$ is a quasi-projective linear scheme over $k$.
\end{lemma}

We assume that $X$ is quasi-projective in order to ensure that
the cyclic product $Z^pX$ is a scheme. If we worked with algebraic
spaces throughout, then the assumption of quasi-projectivity
would be unnecessary.

\begin{proof}
We start by showing that for any representation $V$ of $\Z/p$
over $k$, the quotient variety $V/(\Z/p)$ is a linear scheme,
following \cite[proof of Lemma 8.1]{Totarochow}.
We use induction on the dimension of $V$. We can assume that
$\Z/p$ acts nontrivially on $V$. Then we can write $V=
W\oplus L$, where $L$ is a nontrivial 1-dimensional
representation of $\Z/p$. The quotient variety $V/(\Z/p)$
has a closed subvariety $W/(\Z/p)$, which is a linear scheme
by induction. The open complement is a vector bundle
(with fiber $W$) over $(L-0)/(\Z/p)\cong A^1-0$. A direct
calculation shows that this vector bundle is trivial.
So the open complement is isomorphic to $W\times (A^1-0)$,
which is a linear scheme. Thus $V/(\Z/p)$ is a linear scheme over $k$,
completing the induction.

Next, let $Y$ be a closed subscheme of a scheme $X$ over $k$, and let $U=X-Y$.
Then the cyclic product scheme $Z^pX$ is the disjoint union
(as a set) of $Z^pY$, $Z^pU$, and various products $Y^a\times U^{p-a}$
for $0\leq a\leq p$. Suppose that $X$, $Y$, and $U$ are linear schemes
over $k$. Then all products $Y^a\times U^{p-a}$ are linear
schemes. As a result, if any two of $Z^pX$, $Z^pY$,
and $Z^pU$ are linear schemes, then so is the third. By the inductive
definition of linear schemes, it follows that for every linear scheme
$X$ over $k$, $Z^pX$ is a linear scheme over $k$.
\end{proof}

Let $G$ be an affine group scheme of finite
type over a field $k$. We say that $BG$ can be {\it approximated
by linear schemes }over $k$ if, for every natural number $r$,
there is a representation $V$ of $G$ and a closed $G$-invariant subset $S$
of codimension at least $r$ in $V$
such that $G$ acts freely on $V-S$ and
$(V-S)/G$ is a linear scheme over $k$.
If $BG$ can be approximated by linear schemes,
then $BG$ is mixed Tate. Indeed, for each $r$, $V$, $S$ as just mentioned,
the compactly supported motive of the quotient stack $S/G$
is in the subcategory $(E_{\dim(S)-\dim(G)})^{\perp}$,
by Lemma \ref{dimbound}. Write $V/G$ for the quotient stack. Then
it follows that the cone of the morphism
$$M^c(BG)\cong M^c(V/G)(-\dim (V))[-2\, \dim(V)]\arrow M^c(V-S)/G(-\dim(V))
[-2\, \dim(V)]$$
lies in $(E_{\dim(S)-\dim(V)-\dim(G)})^{\perp}$, hence in
$(E_{-r-\dim(G)})^{\perp}$. Since we assumed that $r$ can be arbitrarily large,
Corollary \ref{approx} gives that $M^c(BG)$ is mixed Tate.

For a group $G$, the {\it wreath product }$\Z/p\wr G$ means
the semidirect product $\Z/p\ltimes G^p$, with $\Z/p$ cyclically
permuting the copies of $G$.

\begin{lemma}
\label{wreathlinear}
Let $k$ be a field of characteristic not $p$ that 
contains the $p$th roots of unity. Let $G$ be an affine group
scheme over $k$ such that $BG$ can be approximated by linear schemes
over $k$. Then $B(\Z/p\wr G)$ can be approximated
by linear schemes over $k$, and hence is mixed Tate.
\end{lemma}

\begin{proof}
Let $V$ be a representation of $G$ over $k$.
Then $V^{\oplus p}$ can be viewed as a representation of
$\Z/p\wr G$, where $\Z/p$ permutes the copies of $V$. If the quotients
make sense, then we have $V^{\oplus p}/(\Z/p\wr G)=Z^p(V/G)$.
It follows that if $BG$ can be approximated by linear schemes $Y$,
then $B(\Z/p\wr G)$ is approximated by the schemes $Z^pY$, which
are linear schemes by Lemma \ref{cyclicproduct}.
\end{proof}

\begin{corollary}
\label{wreathexamples}
Let $G$ be a group scheme over a field $k$ that
satisfies one of the following assumptions. Then $BG$ is mixed Tate
in $DM(k;\Z)$.

\begin{enumerate}
\item $G$ is the multiplicative group $G_m$.

\item $G$ is a finite abelian group of exponent $e$
viewed as an algebraic group over $k$, $e$ is invertible in $k$,
and $k$ contains the $e$th roots of unity.

\item $G$ is an iterated
wreath product\index{wreath product}
$\Z/p\wr \cdots\wr \Z/p\wr G_m$ over $k$,
$p$ is invertible in $k$, and $k$ contains the $p$th roots of unity.

\item $G$ is an iterated
wreath product $\Z/p\wr \cdots\wr \Z/p\wr A$ for a finite abelian
group $A$ of exponent $e$, viewed as an algebraic group over $k$.
Also, $p$ and $e$ are invertible in $k$ and $k$ contains
the $p$th and $e$th roots of unity.
\end{enumerate}
\end{corollary}

\begin{proof}
In all these cases, $BG$ can be approximated by linear schemes over $k$
and hence is mixed Tate. First, $BG_m$ can be approximated
by the schemes $(A^n-0)/G_m=\P^{n-1}$ over $k$ as $n$ increases.
These are linear schemes.
Next, when $A$ is a finite abelian group of exponent $e$
such that $e$ is invertible in $k$ and $k$ contains
the $e$th roots of unity, then $A$ is isomorphic to a product
of the group schemes $\mu_r$ over $k$. The classifying
space $B\mu_r$ can be approximated by the schemes $(A^n-0)/\mu_r$
as $n$ increases, where $\mu_r$ acts by scalars. This scheme
is the total space of the line bundle $O(r)$ minus the zero section
over $\P^{n-1}$, and hence is a linear scheme.
So $BA$ can be approximated by linear schemes, under our assumption
on $k$. Finally, the statements on wreath products follow
from Lemma \ref{wreathlinear}.
\end{proof}

\begin{theorem}
\label{symmetric}
Let $n$ be a positive integer, and let 
$k$ be a field of characteristic zero that contains
the $p$th roots of unity for all primes $p$ dividing $n$.
Then the symmetric group $S_n$
is mixed Tate over $k$ (with $\Z$ coefficients).
\end{theorem}

\begin{proof}
Let $p$ be a prime number. A Sylow $p$-subgroup $H$ of $G=S_n$ is a product
of iterated wreath products $\Z/p\wr\cdots\wr\Z/p$. By Corollary
\ref{wreathexamples}, $BH$ is mixed Tate in $DM(k;\Z)$, hence
in $DM(k;\Z_{(p)})$ by Lemma \ref{localcoefficients}. By Lemma
\ref{sylowtate}, $BG$ is mixed Tate in $DM(k;\Z_{(p)})$.
Since this holds for all prime numbers $p$, $BG$ is mixed
Tate in $DM(k;\Z)$ by Lemma \ref{localcoefficients}.
\end{proof}

\begin{theorem}
\label{gl}
Let $n$ be a positive integer, $q$ a power of a prime number $p$,
and $l$ a prime number different from $p$. Let $r$ be the order
of $q$ in $(\Z/l)^*$, and let $\nu$ be the $l$-adic order of
$q^r-1$. If $l=2$, assume that $q\equiv 1\pmod{4}$.
Let $k$ be a field of characteristic not $l$ that contains
the $l^{\nu}$ roots of unity. Then the finite group $GL(n,\F_q)$
is mixed Tate in $DM(k,\Z_{(l)})$.
\end{theorem}

\begin{proof}
If $l$ is odd, or if $l=2$ and $q\equiv 1\pmod{4}$,
then a Sylow $l$-subgroup of $GL(n,\F_q)$ is a product of wreath
products $\Z/l\wr\cdots\wr \Z/l\wr \Z/l^{\nu}$ \cite{CF,Weir}.
The result
follows from Corollary \ref{wreathexamples} and Lemma \ref{sylowtate}.
\end{proof}

\section{Groups of order 32}

Let $G$ be a $p$-group of order at most $p^4$, for a prime number $p$.
Let $e$ be the exponent of $G$. Let $k$ be a field of characteristic
not $p$ which contains the $e$th roots of unity. Then the Chow ring
of $BG$ consists of transferred Euler classes of representations
\cite[Theorem 11.1]{Totarobook}, and this remains true over every
extension field of $k$. All representations of a subgroup of $G$
over an extension field of $k$ can be defined over $k$, and so it follows
that $G$ has the weak Chow K\"unneth property: $CH^*BG\arrow CH^*BG_E$
is surjective for every extension field $E$ of $k$.

In this section, we show that groups of order 32 also satisfy
the weak Chow K\"unneth property. It follows that the results
after Corollary \ref{BG} are optimal: there are groups of order 64,
and of order $p^5$ for any odd prime number $p$, which do not
have the weak Chow K\"unneth property.

Our proof of the weak Chow K\"unneth property for groups $G$
of order 32 uses the fact that $BG$ is stably rational
for these groups, 
by Chu, Hu, Kang, and Prokhorov \cite{CHKP}.
We do not know how to relate these two properties in general;
as discussed in section \ref{TateBGsect}, they may be equivalent.

\begin{theorem}
\label{order32}
Let $G$ be a group of order 32. Let $e$ be the exponent
of $G$. Let $k$ be a field of characteristic
not $2$ which contains the $e$th roots of unity. Then $BG$ over $k$
satisfies the weak Chow K\"unneth property.
\end{theorem}

\begin{proof}
For every proper subgroup $H$ of $G$, $H$ has order dividing 16,
and so $BH$ over $k$ satisfies
the weak Chow K\"unneth property, as mentioned above.

Let $V$ be a faithful representation of $G$ over $k$. Since $k$ does not have
characteristic 2, $V$ is a direct sum of irreducible representations,
$V=\oplus_{i=1}^c V_i$. Write $P(W)$ for the space of hyperplanes
in a vector space $W$, so that $P(W^*)$ is the space of lines in $W$.
Then $G$ acts on the product of projective spaces
$Y=P(V_1^*)\times \cdots\times P(V_c^*)$. The kernel of the action
of $G$ on $Y$ is the center of $G$, by Schur's lemma. Let $D_Y$
be the closed subset of $Y$ where $G/Z(G)$ does not act freely.
Let $D$ be the union of $\cup_{i=1}^c \oplus_{j\neq i}V_j$ with
the inverse image of $D_Y$ in $V$. Then $D$ is a $G$-invariant
finite union of linear subspaces of $V$, and $D\neq V$.

\begin{lemma}
\label{weakCKrep}
Let $G$ be a $p$-group.
Let $V$ be a faithful representation of $G$
over a field $k$ of characteristic not $p$.
Let $Y$ be the product of projective spaces defined above,
and define $D_Y$ and $D$ as above.
Suppose that the variety $(V-D)/G$ has the weak
Chow K\"unneth property. Also, suppose that for every subgroup
$N\neq G$ that is the stabilizer
of some intersection of irreducible components
of $D$ (as a set), $BN$ has the weak Chow K\"unneth property.
Then $BG$ has the weak Chow K\"unneth property.
\end{lemma}

Lemma \ref{weakCKrep} is analogous to Theorem \ref{reduceBG}
on the mixed Tate property,
but the argument for the weak Chow K\"unneth property is simpler.

\begin{proof}
(Lemma \ref{weakCKrep})
By the localization sequence for Chow groups of quotient stacks
\cite[section 2.7]{EG},
if a quotient stack $X$ over $k$ has the weak Chow K\"unneth property,
then so does every open substack of $X$. Also, if a closed substack
$S$ of $X$ and $X-S$ both have the weak Chow K\"unneth property,
then so does $X$. We sometimes write CK for Chow K\"unneth.

We need some variants of these statements. For an integer $a$,
say that a quotient stack
$X$ has the weak CK property {\it in dimension at least $a$ }if
$CH_iX\arrow CH_iX_E$ is surjective for all fields $E/k$ and all $i\geq a$.
Also, say that $X$ has the weak CK property {\it in codimension
$b$ }if $X$ has the weak CK property in dimension
at least $\dim(X)-b$. By the localization sequence for Chow groups,
if $X$ has the weak CK property in codimension $b$,
then so does any open substack of the same dimension as $X$. Also,
if a closed substack $S$ of $X$ and $X-S$ both have the weak
CK property in dimension at least $a$, then so does $X$.

To prove the lemma, we show by induction on $b$ that $BG$
has the weak Chow K\"unneth property in codimension $b$
for all $b$. This is clear for $b=-1$. Suppose that $BG$
has the weak CK property in codimension $b$. To show that $BG$
has the weak CK property in codimension $b+1$, it is equivalent
to show that the stack $V/G$ (a vector bundle over $BG$)
has the weak CK property in codimension $b+1$. 
We are assuming that the variety $(V-D)/G$ has the weak
CK property. Its complement in the stack $V/G$ is a finite disjoint union
of locally closed substacks of the form $U/N$, where $U$ is an open subset
of a linear subspace $W\subsetneqq V$ and $N$ is the stabilizer in $G$
of $W$ as a set. 
If $N\neq G$, then we are assuming that $BN$
has the weak CK property. So the stack $W/N$ (a vector bundle
over $BN$) has the weak CK property, and hence the open
substack $U/N$ has the weak CK property. On the other hand,
if $N=G$,
then $BG$ has the weak CK property in codimension $b$ by the inductive
assumption, and so the stack $W/G$ and its open substack
$U/G$ have the weak CK property in codimension $b$. Here $W$ 
has codimension $>0$ in $V$.
We conclude that the stack $V/G$ has the weak CK property
in codimension $b+1$, thus completing the induction. So $BG$
has the weak CK property.
\end{proof}

We continue the proof of Theorem \ref{order32}.
Let $G$ be a group of order 32. Let $e$ be the exponent of $G$,
and let $k$ be a field of characteristic
not 2 that contains the $e$th roots of unity.
If $G$ is not isomorphic to $(\Z/2)^5$, then $G$ has a faithful
complex representation $V$ of dimension 4. (This can be checked
using the free group-theory program GAP \cite{GAP},
or by the methods of Cernele-Kamgarpour-Reichstein \cite[proof
of Lemma 13]{CKR}.) 
The group $(\Z/2)^5$ has the weak CK property
as we want, and so we can assume that $G$ has a faithful
representation of dimension 4.
The representation theory of $G$ is the ``same''
over $k$ as over $\C$, and so $G$ has a faithful representation
$V$ of dimension 4 over $k$. As above, write
$V=\oplus_{i=1}^c V_i$, and $Y=P(V_1^*)\times \cdots \times P(V_c^*)$.

By Lemma \ref{weakCKrep}, $BG$ over $k$ has the weak CK property
if the $k$-variety $(V-D)/G$ of dimension 4 has the weak CK property.
The variety $(V-D)/G$ is
a principal bundle over $(Y-D_Y)/(G/Z(G))$, with fiber
$(G_m)^c/Z(G)\cong (G_m)^c$. (The representation $V$ gives an inclusion
of the center $Z(G)$ into $(G_m)^c$, which describes the scalar
by which an element of the center acts on each irreducible summand $V_i$.)
So the pullback homomorphism
$$CH^*(Y-D_Y)/(G/Z(G))\arrow CH^*(V-D)/G$$
is surjective.  The variety $(Y-D_Y)/(G/Z(G))$
has dimension $4-c$, which is at most 3. 
As a result, $CH^*(V-D)/G$ is concentrated in degrees
at most $4-c$.

The group $CH^iBG$ is always generated by Chern classes of representations
for $i\leq 2$ \cite[Theorem 3.2]{Totarochow}. All representations of $G$
over an extension field of $k$ can be defined over $k$, and so $BG$
has the weak CK property in codimension 2 (meaning that
$CH^iBG\arrow CH^iBG_E$ is surjective for any $i\leq 2$ and any
field extension $E$ of $k$). So the stack $V/G$ and hence the variety
$(V-D)/G$ have the weak CK property in codimension 2.
If $c\geq 2$, meaning that $V$ is reducible, then $CH^*(V-D)/G$
is concentrated in degrees at most 2 by the previous paragraph.
So $(V-D)/G$ has the weak CK property and we are done.

There remains the case where $c=1$, that is, where $G$ has a faithful
irreducible representation $V$ of dimension 4 over $k$. In this case,
$CH^*(V-D)/G$ is concentrated in degrees at most 3, and this remains
true over any extension field of $k$. We know that $(V-D)/G$ has the
weak CK property in codimension 2, and we want to show that it has
the weak CK property in codimension 3.

We use the fact that $BG$ is stably rational over $k$ for all groups
$G$ of order 32, under our assumption on $k$,
by Chu, Hu, Kang, and Prokhorov \cite{CHKP}. This means that the variety
$(V-D)/G$ is stably rational over $k$. Since $(V-D)/G$ is a principal
$G_m$-bundle
over the 3-fold $(Y-D_Y)/(G/Z(G))$, that 3-fold is also stably
rational over $k$. It follows
that $CH^3(Y-D_Y)/(G/Z(G))$ is generated by a $k$-rational point
on  $(Y-D_Y)/(G/Z(G))$,
and this remains true over every extension field of $k$.
So $(Y-D_Y)/(G/Z(G))$ has the weak CK property in codimension 3.
By the surjection $CH^3(Y-D_Y)/(G/Z(G))\surj CH^3(V-D)/G$,
which remains true over every extension field of $k$,
$(V-D)/G$ has the weak CK property in codimension 3. By what we have
said, this completes the proof that $BG$ has the weak CK property.
\end{proof}

% Omit these bibliography lines if there's no bibliography.

\small \sc UCLA Mathematics Department, Box 951555,
Los Angeles, CA 90095-1555

totaro@math.ucla.edu
\end{document}